\documentclass[11pt]{article}\textwidth 160mm\textheight 235mm
\usepackage{amsfonts}
\usepackage{amssymb}
\usepackage{graphicx}
\usepackage{amsmath}
\usepackage{amsthm}
\usepackage{enumitem}
\usepackage{tikz}
\usepackage{cancel}
\usepackage{todonotes}
\usetikzlibrary{matrix}
\usetikzlibrary{arrows,shapes}
\usepackage{cite}
\usepackage{hyperref}
\hypersetup{colorlinks=true, urlcolor=blue}
\expandafter\let\expandafter\oldproof\csname\string\proof\endcsname
\let\oldendproof\endproof
\renewenvironment{proof}[1][\proofname]{%
  \oldproof[\ttfamily \scshape \bf #1. ]%
}{\oldendproof}

\def\S{{\mathbb{S}}}

\def\B{\mathbb{B}}
\def\R{{\rm I\!R}}
\def\N{{\rm I\!N}}
\def\ox{\bar{x}}
\def\oy{\bar{y}}
\def\oz{\bar{z}}
\def\ov{\bar{v}}
\def\ow{\bar{w}}

\def\ss{\scriptsize }

\def\ve{\varepsilon}

\def\X{{\mathbb X}}
\def\Y{{\mathbb Y}}

\def\emp{\emptyset}

\def\Lm{{\Lambda}}
\def\tto{\rightrightarrows}

\def\d{{\rm d}}
\def\sub{\partial}

\def\Hat{\widehat}

\def\Bar{\overline}
\def\ra{\rangle}
\def\la{\langle}
\def\ve{\varepsilon}

\def\eig {\mbox{\rm eig}\,}
\def\tr{\mbox{\rm tr}\,}

\def\gph{\mbox{\rm gph}\,}
\def\epi{\mbox{\rm epi}\,}

\def\dom{\mbox{\rm dom}\,}

\def\ker{\mbox{\rm ker}\,}

\def\dn{\downarrow}

\def\ph{\varphi}

\def\emp{\emptyset}
\def\st{\stackrel}
\def\oR{\Bar{\R}}
\def\lm{\lambda}

\def\dd{\delta}
\def\al{\alpha}

\def\sm{\hbox{${1\over 2}$}}
\def\rsm{\hbox{${r\over 2}$}}
\def\lsm{\hbox{${\ell\over 2}$}}

\def\sce{\setcounter{equation}{0}}

\def\verl{ \;\rule[-0.4mm]{0.2mm}{0.27cm}\;}
\def\verll{ \;\rule[-0.7mm]{0.2mm}{0.37cm}\;}
\def\verlm{ \;\rule[-0.5mm]{0.2mm}{0.33cm}\;}

\begin{document}
\vspace*{0.5in}
\begin{center}
{\bf TWICE EPI-DIFFERENTIABILITY OF EXTENDED-REAL-VALUED FUNCTIONS WITH APPLICATIONS IN  COMPOSITE OPTIMIZATION  }\\[1 ex]
ASHKAN MOHAMMADI\footnote{Department of Mathematics, Wayne State University, Detroit, MI 48202, USA (ashkan.mohammadi@wayne.edu).}
 and M. EBRAHIM SARABI\footnote{Department of Mathematics, Miami University, Oxford, OH 45065, USA (sarabim@miamioh.edu).}
\end{center}
\vspace*{0.05in}
\small{\bf Abstract.}  The paper is devoted to the study of the twice epi-differentiablity of extended-real-valued functions, with an emphasis on
functions satisfying  a certain composite  representation. This will be conducted under  parabolic regularity, a second-order regularity condition that was
recently utilized in \cite{mms2} for second-order variational analysis of constraint systems.  Besides justifying the twice epi-differentiablity of 
composite functions, we obtain precise formulas for their  second subderivatives under the metric subregularity constraint qualification.  
The latter allows us to derive  second-order optimality conditions for a large class of composite optimization problems.   \\[1ex]

\newtheorem{Theorem}{Theorem}[section]
\newtheorem{Proposition}[Theorem]{Proposition}
\newtheorem{Remark}[Theorem]{Remark}
\newtheorem{Lemma}[Theorem]{Lemma}
\newtheorem{Corollary}[Theorem]{Corollary}
\newtheorem{Definition}[Theorem]{Definition}
\newtheorem{Example}[Theorem]{Example}
\newtheorem{Algorithm}[Theorem]{Algorithm}
\renewcommand{\theequation}{{\thesection}.\arabic{equation}}
\renewcommand{\thefootnote}{\fnsymbol{footnote}}

\normalsize

\section{Introduction}\sce
This paper aims to provide a systematic study of the  twice epi-differentiability of extend-real-valued functions in finite dimensional spaces. 
In particular, we pay special attention  to  the composite optimization problem
\begin{equation}\label{comp}
\mbox{minimize} \;\; \ph(x)+ g(F(x))\quad \mbox{over all}\;\; x\in \X,
\end{equation}
where $\ph:\X\to \R$ and $F:\X\to \Y$ are   twice differentiable   and  $g:\Y\to \oR:=(-\infty, +\infty]$ is 
a lower semicontinuous (l.s.c.)\ convex function and where $\X$ and $\Y$ are two finite dimensional spaces, 
and verify the  twice epi-differentiability of the objective function in \eqref{comp} under verifiable assumptions.
The composite optimization problem \eqref{comp} encompasses major classes of   constrained and  composite optimization problems including classical nonlinear programming problems,
second-order cone and semidefinite programming problems, eigenvalue optimizations problems \cite{t2}, and fully amenable composite optimization problems \cite{r89}, 
see Example~\ref{ex04} for more detail. Consequently, the composite  problem \eqref{comp} provides  
 a unified framework to study second-order variational properties, including the twice epi-differentiability and second-order optimality 
 conditions,  of  the aforementioned  optimization problems. As argued below, the twice epi-differentiability 
carries vital second-order information for extend-real-valued functions and therefore plays an important role 
 in modern second-order variational analysis. 
  
 A lack of an appropriate second-order generalized derivative for nonconvex  extended-real-valued functions was the main driving force for Rockafellar
 to introduce in \cite{r85} the concept of the twice epi-differentiability for such functions. Later, in his landmark paper \cite{r89}, Rockafellar justified this property for an important 
 class of functions, called {\em fully amenable}, that includes  nonlinear programming problems but does not go far enough to cover other major classes of constrained and composite optimization  problems. 
Rockafellar's results were extended in \cite{c91,i91} for composite functions appearing in \eqref{comp}. However, these extensions 
were achieved under a restrictive assumption on the second subderivative, which does not hold for constrained optimization problems.
Nor does this condition hold for other major composite functions related to  eigenvalue optimization problems; see \cite[Theorem~1.2]{t2} for more detail.  Levy in \cite{l01} obtained upper and lower
 estimates for the second subderivative of the composite function  from \eqref{comp}, but fell short of establishing the twice epi-differentiability 
 for this framework. 

The authors and Mordukhovich observed recently in \cite{mms2} that a second-order regularity, called {\em parabolic regularity} (see Definition~\ref{pre}), 
can play a major role toward the establishment  of the  twice epi-differentiability  for constraint systems, namely when the outer function $g$ in \eqref{comp}
is the indicator function of a closed convex set. This vastly alleviated the difficulty that was often appeared  in the justification  of the  twice epi-differentiability for the latter framework and opened the door
for crucial applications of this concept in theoretical and numerical aspects of optimization. Among these applications, we can list the following:
 \begin{itemize}[noitemsep,topsep=0pt]
\item the calculation of proto-derivatives of subgradient mappings via  the connection between the second subderivative of a function and the proto-derivative of its subgradient mapping (see equation \eqref{gdpd});
\item the calculation of the second subderivative of the augmented Lagrangian function associated with the composite problem \eqref{comp}, which 
allows us to characterize the second-order growth condition for the augmented Lagrangian problem (cf.  \cite[Theorems~8.3 \& 8.4]{mms2});
 \item the validity of the derivative-coderivative inclusion (cf. \cite[Theorem~13.57]{rw}), which has important consequences in parametric optimization; see  \cite[Theorem~5.6]{mms3} for a recent application in the 
 convergence analysis of the sequential quadratic programming (SQP) method for constrained optimization problems.
 \end{itemize}

In this paper, we continue the path, initiated in \cite{mms2} for constraint systems, and show that the twice epi-differentiability 
of the objective function in \eqref{comp} can be guaranteed under parabolic regularity. To achieve this goal, we demand that  
 the outer function $g$ from \eqref{comp} be locally Lipschitz continuous relative to its domain; see the next section for the precise definition 
 of this concept.  Shapiro in \cite{sh03} used a similar condition but in addition assumed that this function is finite-valued. 
 The latter does bring certain restrictions for \eqref{comp} by excluding constrained problems as well as piecewise linear-quadratic 
 composite problems. 
 As shown in Example~\ref{ex04}, major classes of constrained and composite optimization problems satisfy this Lipschitzian condition.
 However, some composite problems such as the spectral abcissa minimization (cf. \cite{blo}), namely the problem of minimizing the largest real parts of  eigenvalues,  can not be covered by \eqref{comp}.
 
 The rest of the paper is organized as follows. Section~\ref{sect02} recalls important notions of variational analysis that are used throughout  this paper. 
Section~\ref{sect03} begins with the definition of parabolic regularity of extended-real-valued functions.
Then we justify that
parabolic regularity amounts to a certain  duality relationship between the second subderivative and parabolic subderivative. 
Employing this, we  show that the twice epi-differentiability of extended-real-valued functions can be guaranteed if they are  parabolically regular and 
parabolic epi-differentiable. Section~\ref{sect04} is devoted to  important second-order variational properties of 
parabolic subderivatives. In particular, we establish a chain rule for parabolic subderivatives of   composite functions  in \eqref{comp} under the metric
subregularity constraint qualification.  In Section~\ref{sect05}, we establish  chain rules for the parabolic regularity and for 
the second subderivative of composite functions, and consequently establish their twice epi-differentiability.  Section~\ref{sect06} deals with  important applications of our results in  second-order
optimality conditions for the composite optimization problem \eqref{comp}. We close the paper by achieving a characterization of 
the strong metric subregularity of the subgradient mapping of the objective function in \eqref{comp} via the second-order sufficient condition for this problem.

In what follows,   $\X$ and $\Y$ are  finite-dimensional  Hilbert spaces equipped
with a scalar product $\la\cdot,\cdot\ra$ and its induced norm $\|\cdot\|$.
 By $\B$ we 
 denote the closed unit ball in the space in question and by $\B_r(x):=x+r\B$ the closed ball centered at $x$ with radius $r>0$. 
 For any set $C$ in $\X$, its indicator function is defined by $\dd_C(x)=0$ for $x\in C$ and $\dd_C(x)=\infty$ otherwise. We denote by $d(x,C)$  the distance between $x\in \X$ and a set $C$.
 For $v\in \X$, the subspace $\{w\in \X|\, \la w,v\ra=0\}$ is denoted by $\{v\}^\bot$. 
We write $x(t)=o(t)$ with $x(t)\in \X$ and $t>0$ to mean that ${\|x(t)\|}/{t}$ goes to $0$ as $t\dn 0$.
Finally, we denote by $\R_+$ (respectively,  $\R_-$) the set of non-negative (respectively, non-positive) real numbers.
\vspace{-0.1 in}

\section{ Preliminary Definitions in Variational Analysis}\sce \label{sect02}

In this section we first briefly review basic constructions of variational analysis and generalized differentiation employed in the paper;
see \cite{mor18,rw} for more detail.
 A family of sets $C_t$  in $\X$ for $t>0$ converges to a set $C\subset \X$ if $C$ is closed and 
$$
\lim_{t\dn 0}d(w,C_t)=d(w,C)\quad \mbox{for all}\;\;w\in \X.
$$
Given a nonempty set $C\subset\X$ with $\ox\in C$, the  tangent cone $T_ C(\ox)$ to $C$ at $\ox$ is defined by
\begin{equation*}\label{2.5}
T_C(\ox)=\big\{w\in\X|\;\exists\,t_k{\dn}0,\;\;w_k\to w\;\;\mbox{ as }\;k\to\infty\;\;\mbox{with}\;\;\ox+t_kw_k\in C\big\}.
\end{equation*}
We say a tangent vector $w\in T_C(\ox)$ is {\em derivable} if there exist a constant  $\ve>0$ and an arc $\xi:[0,\ve]\to C$ such that  $\xi(0)=\ox$ and $\xi'_+(0)=w$, where $\xi'_+$ signifies the right derivative of $\xi$
at $0$, defined by 
$$
\xi'_+(0):=\lim_{t\dn 0}\frac{\xi(t)-\xi(0)}{t}.
$$
The set $C$ is called geometrically derivable at $\ox$ if every tangent vector $w$ to $C$ at $\ox$ is derivable. The geometric derivability of $C$ at $\ox$ can be equivalently described by the sets $[C-\ox]/{t}$ 
converging to $T_C(\ox)$  as $t\dn 0$. Convex sets are 
important examples of geometrically derivable sets. The second-order tangent set to $C$ at $\ox$ for a tangent vector $w\in T_C(\ox)$ is given by
\begin{equation*}\label{2tan}
T_C^2(\ox, w)=\big\{u\in\X|\;\exists\,t_k{\downarrow}0,\;\;u_k\to u\;\;\mbox{ as }\;k\to\infty\;\;\mbox{with}\;\;\ox+t_kw+\frac{1}{2} t_k^2 u_k\in C\big\}.
\end{equation*}
 A set $C$ is said to be {parabolically derivable} at $\ox$ for $w$ if $T_C^2(\ox, w)$ is nonempty and for each $u\in T_C^2(\ox, w)$ there are $\ve>0$ and an are $\xi:[0,\ve]\to C$ with $\xi(0)=\ox$, 
$\xi'_+(0)=w$, and $\xi''_+(0)=u$, where 
$$
\xi''_+(0):=\lim_{t\dn 0}\frac{\xi(t)-\xi(0)-t\xi'_+(0)}{\frac{1}{2}t^2}.
$$
It is well known that if $C\subset \X$ is convex and parabolically  derivable at $\ox$ for  $w$, then  the second-order tangent set $T_C^2(\ox, w)$
is a nonempty convex set in $\X$ (cf. \cite[page~163]{bs}).
Given the function $f:\X \to \oR:= (-\infty, \infty]$, its domain and epigraph are defined, respectively, by 
$$
\dom f =\big\{ x \in \X|\; f(x) < \infty \big \}\quad \mbox{and}\quad \epi f=\big \{(x,\al)\in \X\times \R|\, f(x)\le \al\big \}.
$$
 The regular subdifferential  of $f$  at $\ox\in \dom f$ is defined by 
\begin{equation*}
\Hat \sub f(\ox)=\big\{v\in \X \;|\; \liminf_{x\to \ox} \frac{f(x)- f(\ox)-\la v,x-\ox\ra}{\|x-\ox\|}\ge 0\big\}.
\end{equation*}
The  subdifferential  of $f$ at $\ox$ is given by 
\begin{eqnarray*}\label{2.4}
\sub f(\ox)=\big\{v\in \X\; | \;\exists\,x_k\st{f}{\to}\ox,\;\;v_k\to v\;\;\mbox{with}\;\;v_k\in \Hat \sub f(x_k)\big \},
\end{eqnarray*}
where $x_k\st{f}{\to}\ox$ stands for $x_k\to \ox$ and $f(x_k)\to f(\ox)$.
We say that  $v \in \X$ is  a proximal subgradient of $f$ at $\ox$  if there exists $r \in \R_+$ and   a neighborhood $U$ of $\ox$ such that for all $x \in U$ we have 
\begin{equation}\label{proxs}
f(x) \geq f(\ox) + \la v , x - \ox \ra - \frac{r}{2} \|x - \ox \|^2. 
\end{equation}
The set of all such $v$ is called the proximal subdifferential of $f$ at $\ox$ and is denoted by $\sub ^p f(\ox)$. 
By definitions, it is not hard to obtain  the inclusions $\sub^p f(\ox)\subset \Hat \sub f(\ox) \subset \sub f(\ox)$. 
Given a nonempty set $C\subset \X$, the proximal and regular normal cones to $C$ at $\ox\in C$ are defined, respectively, by 
$$
N^p_C(\ox):=\sub ^p \dd_C(\ox)\quad \mbox{and}\quad \Hat N_C(\ox):=\Hat\sub  \dd_C(\ox).
$$
Similarly, we define the (limiting/Mordukhovich) normal cone of $C$ at $\ox$ by $N_C(\ox):=\sub  \dd_C(\ox)$.
Consider a set-valued mapping $S:\X\tto\Y $ with its domain and graph  defined, respectively, by 
$$ \dom S=\big\{x\in\X |\;S(x)\ne\emp\big\}\quad \mbox{and}\quad \gph S=\big\{(x,y)\in\X\times\Y |\;y\in S(x)\big\}.$$
 The {graphical derivative} of $S$ at $(\ox,\oy)\in \gph S$ is defined by 
\begin{equation*}\label{gder}
DS(\ox , \oy)(w)=\big\{v\in\Y|\;(w,v)\in T_{\scriptsize{\gph S}}(\ox,\oy)\big\},\quad w\in\X.
\end{equation*}
Recall that a set-valued mapping $S\colon\X\tto\Y$ is {metrically regular} around  $(\ox,\oy)\in\gph S$ if there are constants  $\kappa \in \R_+$  and $\ve>0 $ such that the distance estimate
\begin{equation*}\label{metreq}
d\big(x,S^{-1}(y)\big)\le \kappa  \,d\big(y, S(x)\big)\;\;\mbox{for all}\quad (x,y)\in \B_\ve(\ox,\oy)
\end{equation*}
holds. When  $y=\oy$ in the above estimate,  the mapping $S$ is called metrically subregular at $(\ox,\oy)$.
The  set-valued mapping $S$ is called strongly metrically subregular at $(\ox,\oy)$ if there are
 a constant $\kappa\in \R_+$ and a neighborhood $U$ of $\ox$ such that the estimate 
 $$
 \|x-\ox\|\le \kappa\, d(\oy, S(x))\;\; \mbox{for all}\;\;x\in U
 $$
holds. It is known (cf. \cite[Theorem~4E.1]{dr}) that the set-valued mapping $S$ is strongly metrically subregular at $(\ox,\oy)$ if and only if 
we have 
\begin{equation}\label{sms8}
0\in DS(\ox, \oy)(w)\implies w=0.
\end{equation}

Given  a function  $f:\X \to \oR$ and  a point $\ox$ with $f(\ox)$ finite, the subderivative function $\d f(\ox)\colon\R^n\to[-\infty,\infty]$ is defined by
$$
{\mathrm d}f(\ox)(\ow)=\liminf_{\substack{
   t\dn 0 \\
  w\to \ow
  }} {\frac{f(\ox+tw)-f(\ox)}{t}}.
$$
Define the parametric  family of 
second-order difference quotients for $f$ at $\ox$ for $\ov\in \X$ by 
\begin{equation*}\label{lk01}
\Delta_t^2 f(\bar x , \ov)(w)=\dfrac{f(\ox+tw)-f(\ox)-t\langle \ov,\,w\rangle}{\frac{1}{2}t^2}\quad\quad\mbox{with}\;\;w\in \X, \;\;t>0.
\end{equation*}
If $f(\ox)$ is finite, then the {second subderivative} of $f$ at $\ox$ for $\ov$   is given by 
\begin{equation*}\label{ssd}
\d^2 f(\bar x , \ov)(w)= \liminf_{\substack{
   t\dn 0 \\
  w'\to w
  }} \Delta_t^2 f(\ox , \ov)(w'),\;\; w\in \X.
\end{equation*}

 Below, we collect  some important properties of the second subderivative  that are 
used throughout this paper. Parts (i) and (ii) were taken from \cite[Proposition~13.5]{rw} and part (iii) was recently observed in \cite[Theorem~4.1(i)]{mms1}.
\begin{Proposition}[properties of second subderivative]\label{ssp} Let $f:\X \to \oR$ and  $(\ox,\ov)\in \X\times \X$ with $f(\ox)$ finite. Then the following conditions hold:
 \begin{itemize}[noitemsep,topsep=0pt]
\item [\bf{(i)}]  the second subderivative $\d^2 f(\ox,\ov)$ is a lower semicontinuous {\rm(}l.s.c.{\rm)}\ function;
\item [\bf{(ii)}] if $\d^2 f(\bar x ,\ov)$ is a proper function, meaning that $\d^2 f(\bar x, \ov)(w)>-\infty$ for all $w\in \X$ and its effective domain, defined by 
 $$
 \dom \d^2 f(\bar x, \ov)=\big\{w\in \X |\, \d^2 f(\bar x, \ov)(w)<\infty\big\},
 $$
 is nonempty, then we always have the inclusion 
 $$
  \dom \d^2 f(\bar x, \ov)\subset \big\{w\in \X |\, {\d}f(\ox)(w)=\la \ov,w\ra\big\};
 $$
 \item [\bf{(iii)}] if $\ov\in \sub^p f(\ox)$, then for any $w\in \X$ we have $\d^2 f(\ox , \ov)(w)\ge -r\|w\|^2$, where 
 $r\in \R_+$ is taken from \eqref{proxs}. In particular, $\d^2 f(\ox , \ov)$ is a proper function.

 \end{itemize}
\end{Proposition}

Following \cite[Definition~13.6]{rw}, a function $f:\X \to \oR$ is said to be {twice epi-differentiable} at $\bar x$ for $\ov\in\X$, with $f(\ox) $ finite, 
if the sets $\epi \Delta_t^2 f(\bar x , \ov)$ converge to $\epi \d^2 f(\bar x,\ov)$ as $t\downarrow 0$. The latter means by  
\cite[Proposition~7.2]{rw} that for every sequence $t_k\downarrow 0$ and every $w\in\X$, there exists a sequence $w_k \to w$ such that
\begin{equation}\label{dedf}
\d^2 f(\bar x,\ov)(w) = \lim_{k \to \infty} \Delta_{t_k}^2 f(\ox , \ov)(w_k).
\end{equation}

We say that a  function $f: \X \to \oR$ is Lipschitz continuous around $\ox $ {\em relative} to $C \subset \dom f$
with constant $\ell \in \R_+ $
if  $\ox \in C$ and there exists  a neighborhood $U$ of $\ox$ such that 
\begin{equation*} \label{lipwrtdomain}
 |f(x)  - f(y ) | \leq  \ell\, \| x - y \|  \quad \mbox{for all }\;       x , y \in U \cap C.    
\end{equation*}
Such a function is called {\em locally} Lipschitz continuous relative to $C$ if for every $\ox\in C$, it is Lipschitz continuous around $\ox $ relative to $C$. 
Piecewise linear-quadratic functions (not necessarily convex) and an indicator function of a nonempty set are important examples of 
functions that are locally Lipschitz continuous relative to their {domains}. 
\begin{Proposition}[domain of subderivatives] \label{dfs} Let $f: \X \to \oR$ be Lipschitz continuous around $\ox $ relative to its domain with constant $\ell\in \R_+$. Then we have $\dom \d f(\ox) = T_{\ss \dom f} (\ox)$.
In particular, for every $w\in T_{\ss \dom f} (\ox)$, the subderivative $\d f(\ox)(w)$ is finite.
\end{Proposition}
\begin{proof} The inclusion $\dom \d f(\ox) \subset T_{\ss \dom f} (\ox)$ results  directly from  the definition. To prove the opposite inclusion, 
pick  $w \in T_{\ss\dom f} (\ox)$. This gives us some sequences $t_k\dn 0$ and $w_k\to w$ such that $\ox+t_kw_k\in \dom f$ for all $k\in \N$. Using this and 
the Lipschitz continuity of $f$ around $\ox$ relative to its domain implies that for all $k$ sufficiently large we have  
\begin{equation}\label{fsdp}
 | \frac{f(\ox+ t_k w_k)- f(\ox)}{t_k} | \leq \ell \|w_k \|. 
\end{equation}
This clearly yields  $|\d f(\ox)(w)| \leq \ell \|w\|$.  Thus $\d f(\ox)(w)$ is finite and so  $w\in \dom \d f(\ox)$. 
This gives us  the inclusion $T_{\ss \dom f} (\ox)\subset \dom \d f(\ox)$ and hence completes the proof.
\end{proof}
\vspace{-0.1 in}

\section{ Twice Epi-Differetiability of Parabolically Regular Functions}\label{sect03} 
This section aims to delineate conditions under which the twice epi-differenibility of extend-real-valued functions can be established. 
To this end, we appeal to  an important second-order regularity condition, called parabolic regularity, which was recently exploited in \cite{mms2}
to study a similar property for constraint systems. We begin with the definition of this  regularity condition.

\begin{Definition}[parabolic regularity]\label{pre} A function  $f:\X \to \oR$ is parabolically regular at $\ox$ for  $\ov\in \X$ if 
$f(\ox)$ is finite and  if for any $w$ such that  $\d^2 f(\bar x , \ov)(w)<\infty  $, there exist, among the sequences $t_k\dn 0$ and $w_k\to w$ with 
$\Delta_{t_k}^2 f(\bar x , \ov)(w_k) \to \d^2 f(\bar x , \ov)(w)$, those with the additional property that 
\begin{equation}\label{prc}
\limsup_{k\to \infty} \frac{\|w_k-w\|}{t_k}<\infty.
\end{equation}
A nonempty set $C\subset \X$ is said to be parabolically regular at $\ox$ for  $\ov$ if the indicator function $\dd_C$ is parabolically regular at $\ox$ for  $\ov$.
\end{Definition}
Although the notion of parabolic regularity was introduced   first in \cite[Definition~13.65]{rw}, its origin goes back to \cite[Theorem~4.4]{ch0}, where 
Chaney observed a duality relationship between his second-order generalized derivative and the parabolic subderivative, defined in \cite{bz1} by Ben-Tal and Zowe.  
This  duality relationship was derived later  by Rockafellar \cite[Proposition~3.5]{r88} for convex piecewise 
linear-quadratic functions. As shown in Proposition~\ref{parregularity} below, the latter duality relationship is equivalent to the concept of parabolic regularity from Definition~\ref{pre} provided that $\ov$,
appearing in Definition~\ref{pre},  is a proximal subgradient.
A different  second-order regularity was introduced by Bonnans, Comminetti,  and Shapiro \cite[Definition~3]{bcs} for sets, which was later  extended in \cite[Definition~3.93]{bs} for functions.
It is not difficult to see that parabolic regularity is implied by the second-order regularity in the sense of \cite{bcs}; see \cite[Proposition~3.103]{bs} for a proof of this result. 
Moreover,  the example from  \cite[page~215]{bs}  shows that the converse implication may not hold in general.  

We showed in \cite{mms2} that important sets appearing in constrained optimization problems, including polyhedral convex sets, the second-order cone, and 
the cone of positive semidefinite symmetric matrices, are parabolically regular. Below, we add two important classes of functions for which 
this property automatically fulfill. We begin first by  convex piecewise-linear quadratic functions and then 
consider  eigenvalues functions.
While  the former was justified in \cite[Theorem~13.67]{rw}, we provide below a different and simpler proof. 
\begin{Example}[piecewise linear-quadratic functions]\label{plqf}
{\rm Assume that the function $f : \X \to \oR $ with $\X=\R^n$ 
is convex  piecewise linear-quadratic. Recall  that $f$ is called  piecewise linear-quadratic if $\dom f = \cup_{i=1}^{s} C_i$ with $s\in \N$ and $C_i $ being polyhedral convex  sets for $i = 1, \ldots, s$, and if $f$ has a representation of the form
\begin{equation*} \label{PWLQ}
f(x) = \sm \langle A_i x ,x \rangle + \langle a_i ,x \rangle + \alpha_i  \quad \mbox{for all} \quad  x \in C_i,
\end{equation*}
where $A_i$ is an $n \times n$ symmetric matrix, $a_i\in \R^n$, and $\alpha_i\in \R$ for $i = 1, \cdots, s$.
It was proven in \cite[Propsoition~13.9]{rw}
that the second subderivative of $f$ at $\ox$ for $\ov\in \sub f(\ox)$ can be calculated by 
\begin{equation}\label{pwfor}
 \d^2 f(\ox , \ov ) (w)  
=\begin{cases}
\la A_i w, w\ra&\mbox{if}\;\; w\in T_{ C_i}(\ox) \cap \{\ov_i\}^{\bot},\\
\infty&\mbox{otherwise},
\end{cases}
\end{equation}
where $\ov_i:=\ov -A_i \ox -a_i$. To  prove the parabolic regularity of $f$ at $\ox$ for $\ov$, 
pick a vector $w\in \R^n$ with $\d^2 f(\bar x , \ov)(w)<\infty  $. This implies that there is an $i$ with $1\le i\le s$ such that 
$w\in T_{ C_i}(\ox) \cap \{\ov_i\}^{\bot}$. Since $C_i$ is a polyhedral convex set, we conclude from \cite[Exercise~6.47]{rw} that there exists 
an $\ve>0$ such that $\ox+tw\in C_i$ for all $t\in [0,\ve]$. Pick a sequence $t_k\dn 0$ such that $t_k\in [0,\ve]$ and let $w_k:=w$ for all $k\in \N$.
Thus a simple calculation tells us that 
$$
\begin{array}{lll}
\Delta_{t_k}^2 f(\bar x , \ov)(w_k)=\dfrac{f(\ox+t_kw_k)-f(\ox)-t_k\langle \ov,\,w_k\rangle}{\frac{1}{2}t_k^2}\\
~~~~=\dfrac{\frac{1}{2} \langle A_i (\ox+t_kw_k) ,\ox+t_kw_k \rangle + \langle a_i ,\ox+t_kw_k \rangle + \alpha_i - \frac{1}{2} \langle A_i \ox ,\ox \rangle - \langle a_i ,\ox \rangle - \alpha_i -t_k\langle \ov,\,w_k\rangle}{\frac{1}{2}t_k^2}\\
~~~~= \la A_i w, w\ra+\dfrac{ t_k\la w_k, \ov -A_i \ox -a_i\ra}{ \frac{1}{2}t_k^2}=\la A_i w, w\ra,
\end{array}
$$
which in turn  implies by \eqref{pwfor}  that $\Delta_{t_k}^2 f(\bar x , \ov)(w_k) \to \d^2 f(\bar x , \ov)(w)$ as $k\to \infty$. Since \eqref{prc} is clearly 
holds, $f$ is parabolic regular at $\ox$ for $\ov$.
}
\end{Example}

\begin{Example}[eigenvalue functions]\label{eig}
{\rm Let $\X={\S}^n$ be the space of $n\times n$ symmetric real matrices, which is conveniently treated via the inner product
$$
\la A,B\ra:=\tr AB
$$
with $\tr AB$ standing for the sum of the diagonal entries of $AB$. For a matrix $A\in {\S}^n$, we denote by $A^ \dagger$ the Moore-Penrose pseudo-inverse of $A$
and  by $\mbox{eig}\, A=(\lm_1(A),\ldots, \lm_n(A))$
the vector of eigenvalues of $A$ in decreasing order with eigenvalues repeated according to their multiplicity. Given $i\in \{1,\ldots,n\}$, denote by $\ell_i(A)$
the number of eigenvalues that are equal to $\lm_i(A)$ but are ranked before $i$ including   $\lm_i(A)$.  This integer allows us to locate $\lm_i(A)$
in the group of the eigenvalues of  $A$ as follows:
$$
\lm_1(A)\ge \cdots \ge \lm_{i-\ell_i(A)}>\lm_{i-\ell_i(A)+1}(A)= \cdots =\lm_i(A)\ge \cdots\ge \lm_n(A).
$$
The eigenvalue $\lm_{i-\ell_i(A)+1}(A)$, ranking first in the group of eigenvalues equal to $\lm_i(A)$,  is called the {\em leading}
eigenvalue. For any $i\in \{1,\ldots,n\}$, define now the function $\al_i:{\S}^n\to \R$ by 
\begin{equation}\label{fal}
\al_i(A)= \lm_{i-\ell_i(A)+1}(A)+ \cdots +\lm_i(A),\quad A\in \S^n.
\end{equation}
It was proven in \cite[Theorem~2.1]{t2} that $\Hat\sub \al_i(A)=\sub \al_i(A)$ and that the second subderivative of $\al_i$ at $A$ for any $V\in \sub \al_i(A)$
is calculated for every $W\in {\S}^n$ by 
\begin{equation}\label{alss}
\d^2 \al_i(A,V)(W)=\begin{cases}
2\la V, W(\lm_i(A)I_n-A)^{\dagger}W\ra &\mbox{if}\;\; \d\al_i(A)(W)=\la X,W\ra,\\
\infty& \mbox{otherwise},
\end{cases}
\end{equation}
where $I_n$ stands for the $n\times n$ identity matrix. 
Moreover, for any $W\in {\S}^n$ with $\d^2 \al_i(A, H)(W)<\infty$ and any sequence $t_k\dn 0$,  the proof of \cite[Theorem~2.1]{t2} confirms that 
$$
\Delta^2_{t_k} \al_i(A,V)(W_k)\to \d^2 \al_i(A,V)(W)\quad \mbox{with}\;\; W_k:=W-t_kW(\lm_i(A)I_n-A)^{\dagger}W.
$$
This readily verifies \eqref{prc} and thus the functions $\al_i$,  $i\in \{1,\ldots,n\}$, are parabolically regular at $A$ for any $V\in \sub \al_i(A)$.
In particular, for $i=1$, the function $\al_i$ from \eqref{fal} boils down to the maximum eigenvalue function of a matrix, namely 
\begin{equation}\label{maxeig}
\lm_{\ss\max}(A):=\al_1(A)= \lm_1(A),\quad A\in \S^n.
\end{equation}
So the maximum eigenvalue function $\lm_{\max}$ is parabolically regular at $A$ for any $V\in \sub \lm_{\max}(A)$.
This can be said for any leading eigenvalue $\lm_{i-\ell_i(A)+1}(A)$ since we have $\al_i(B)=\lm_{i-\ell_i(A)+1}(B)$ for every matrix $B\in \S^n$ sufficiently  close to $A$.
Another important function related to the eigenvalues of a matrix $A\in \S^n$ is the sum of the first $i$ components of $\eig A$ with $i\in \{1,\ldots,n\}$, namely 
\begin{equation}\label{fsig}
\sigma_i(A)= \lm_{1}(A)+ \cdots +\lm_i(A).
\end{equation}
It is well-known that the functions $\sigma_i$ are convex (cf. \cite[Exercise~2.54]{rw}). Moreover, we have $\sigma_i(A)=\al_i(A)+\sigma_{i-\ell_i(A)}(A)$.
It follows from \cite[Proposition~1.3]{t2} that  $\sigma_{i-\ell_i(A)}$ is  twice continuously differentiable (${\cal C}^2$-smooth) on ${\S}^n$. 
This together with the parabolic regularity of $\al_i$ ensures that $\sigma_i$ are parabolically regular at $A$ for any $V\in \sub \sigma_i(A)$.
}
\end{Example}

To proceed further in this section, we require the concept of the parabolic subderivative, introduced by  Ben-Tal and Zowe in \cite{bz1}. 
Let    $f:\X \to \oR$ and let  $\ox\in \dom f$   and $w\in \X$ with $\d f(\ox)(w)$ finite.  
The {\em parabolic subderivative} of $f$ at $\ox$ for $w$ with respect to $z$ is defined by 
\begin{equation*}\label{lk02}
\d^2 f(\bar x)(w\verl z):= \liminf_{\substack{
   t\dn 0 \\
  z'\to z
  }} \dfrac{f(\ox+tw+\frac{1}{2}t^2 z')-f(\ox)-t\d f(\ox)(w)}{\frac{1}{2}t^2}.
\end{equation*}
Recall from  \cite[Definition~13.59]{rw} that $f$ is called {\em parabolically epi-differentiable} at $\ox$ for $w$ if  
$$\dom \d^2 f(\ox)(w \verl \cdot)=\big\{z\in \X|\,  \d^2 f(\ox)(w \verl z)<\infty \big\}\neq \emptyset,$$
 and for every $z \in \X$ and every sequence $t_k\dn 0$  there exists a  sequences $z_k\to z$ such that
\begin{equation}\label{pepi}
\d^2 f(\bar x)(w\verl z) = \lim_{k\to \infty} \dfrac{f(\ox+t_kw+\frac{1}{2}t_k^2 z_k)-f(\ox)-t_k\d f(\ox)(w)}{\frac{1}{2}t_k^2}.
\end{equation}

The main interest in parabolic subderivatives in this paper lies in its nontrivial connection with second subderivatives. 
Indeed,  it was shown in \cite[Proposition~13.64]{rw} that if the function $f:\X\to \oR$ is finite at $\ox$, then for any pair $(\ov,w)\in \X\times \X$ with  $\d f(\ox)(w)=\la w,\ov\ra$ we always  have 
\begin{equation}\label{pr2}
\d^2 f(\bar x,\ov)(w)\leq \inf_{z\in \X}\big\{\d^2 f(\ox)(w\verl z) -\la z, \ov\ra\big \}.
\end{equation}
As observed below, equality in this estimate amounts to the parabolic regularity of $f$ at $\ox$ for $\ov$.
To proceed, let  $f: \X \to \oR$ and  pick $(\ox,\ov)\in \gph \sub f$. The {\em critical cone} of $f$ at  $(\ox,\ov)$ is defined by  
\begin{equation}\label{krit1}
K_f(\ox,\ov):= \big\{ w \in \X \:|\: \d f(\ox)(w) = \la \ov , w \ra \big\}.
\end{equation}
When $f$ is the indicator function of a set, this definition boils down to the classical definition of the critical cone for sets; see \cite[page~109]{dr}.
It is not difficult to see that the set $K_f(\ox,\ov)$ is a cone in $\X$. Taking into account Proposition~\ref{ssp}(ii), we conclude that the domain of the second subderivative $ \d^2 f (\bar x , \ov ) $
is always included in the critical cone $K_f(\ox,\ov)$ provided that $ \d^2 f (\bar x , \ov ) $ is a proper function. 
The following result provides conditions under which the domain of the second subderivative is the entire critical cone. 

\begin{Proposition}[domain of second subderivatives]\label{dss}
Assume that  $f:\X \to \oR$ is finite at $\ox$ with $\ov \in \sub^p f(\ox)$  and that for every $w \in K_f(\ox,\ov)$ we have $ \dom \d^2 f(\ox)(w\verl \cdot)\neq \emptyset$. 
Then for all $w \in K_f (\ox,\ov) $ we have 
\begin{equation}\label{pri3}
-r \|w\|^2 \leq  \d^2 f (\bar x,\ov ) (w) \leq  \inf_{z\in \X}\big\{\d^2 f(\ox)(w\verl z) -\la z, \ov\ra\big \}  < \infty,
\end{equation} 
where $r\in \R_+$ is a constant satisfying \eqref{proxs}. In particular, we have  $ \dom \d^2 f (\bar x , \ov ) =  K_f (\ox,\ov)$.
\end{Proposition}
\begin{proof} The lower estimate of $\d^2 f (\bar x , \ov )$ in \eqref{pri3} results from Proposition~\ref{ssp}(iii), which readily implies that $\d^2 f (\bar x,\ov ) (0)=0$. This tells us that  the second subderivative $\d^2 f (\bar x,\ov )$
is proper. Employing now Proposition~\ref{ssp}(ii) gives us the inclusion $ \dom \d^2 f (\bar x , \ov )\subset K_f (\ox,\ov)$. 
The upper estimate of $\d^2 f (\bar x , \ov ) (w)$ in \eqref{pri3} directly comes from \eqref{pr2}.
By assumptions,   for any $w \in K_f (\ox,\ov) $, there exists a $z_{w}$ so that $ \d^2 f(\ox)(w\verl z_w)<\infty$. This  guarantees that the {infimum} term in \eqref{pri3} is finite.
Pick $w\in K_f (\ox,\ov) $ and observe from \eqref{pri3} that $\d^2 f (\bar x , \ov ) (w)$ is  finite. This yields  the inclusion 
 $K_f (\ox,\ov)\subset  \dom \d^2 f (\bar x , \ov )$, which   completes the proof. 
\end{proof}

 The following example, taken from \cite[page~636]{rw}, shows the domain of the second subderivative can be the entire set $K_f (\ox,\ov)$
even if the assumption on the domain of the parabolic subderivative  in Proposition~\ref{dss} fails. As shown in the next section, however, this condition is automatically satisfied  for  composite functions appearing in \eqref{comp}.
\begin{Example}[domain of second subderivative] \label{ex01} {\rm Define the function $f:\X\to \oR$ with $\X=\R^2$ by 
$f(x_1,x_2)=|x_2-x_1^{{4/3}}|-x_1^2$.  As argued in \cite[page~636]{rw}, the subderivative and subdifferential of $f$ at $\ox=(0,0)$, respectively, are 
$$
\d f(\ox)(w)=|w_2|\quad \mbox{and}\quad \sub f(\ox)=\big\{v=(v_1,v_2)\in \R^2|\, v_1=0, \,|v_2|\le 1\big\},
$$
where $w=(w_1,w_2)\in \R^2$. It is not hard to see that $\ov=(0,0)\in \sub^p f(\ox)$. Moreover, the second subderivative of 
$f$ at $\ox$ for $\ov$ has a representation of the form 
$$
\d^2 f(\ox, \ov)(w)=\begin{cases}
-2w_1^2&\mbox{if}\;\; w_2=0,\\
\infty&\mbox{if}\;\; w_2\neq 0.
\end{cases}
$$
Using the above calculation tells us that $K_f (\ox,\ov)=\{w=(w_1,w_2)|\, w_2=0\}$. Thus we have $\dom \d^2 f (\bar x , \ov ) =  K_f (\ox,\ov)$.
However, for any $w=(w_1,w_2)\in  K_f (\ox,\ov)$ with $w_1\neq 0$ we have 
$$
\d^2 f(\ox)(w\verl z)=\infty\quad \mbox{for all}\;\; z\in \R^2,
$$
which confirms that the assumption related to the domain of the parabolic subderivative in Proposition~\ref{dss} fails. 
}
\end{Example}

We proceed next by providing an important characterization of the  parabolic regularity that plays a key role in our developments 
in this paper. 

\begin{Proposition}[characterization of parabolic regularity] \label{parregularity} 
Assume that  $f:\X \to \oR$ is finite at $\ox$ with $\ov \in \sub^p f(\ox)$. Then  
the  function $f$ is parabolically regular at $\ox$ for  $\ov $ if and only if we have
\begin{equation}\label{pri1}
\d^2 f (\bar x , \ov ) (w) =  \inf_{z\in \X}\big\{\d^2 f(\ox)(w\verl z) -\la z, \ov\ra\big \} 
\end{equation}
for all $w\in K_f (\ox,\ov)$. Furthermore, for any $w\in \dom \d^2 f (\bar x , \ov )$, there exists a $\oz\in \dom \d^2 f(\ox)(w\verl \cdot)$ such that 
\begin{equation}\label{pri11}
\d^2 f (\bar x , \ov ) (w) = \d^2 f(\ox)(w\verl \oz) -\la \oz, \ov\ra.
\end{equation}
\end{Proposition} 
\begin{proof} It follows from $\ov \in \sub^p f(\ox)$ and  Proposition~\ref{ssp}(ii)-(iii) that the second subderivative  $ \d^2 f (\bar x , \ov ) $ is a proper function and 
\begin{equation}\label{dss1}
\dom \d^2 f (\bar x , \ov ) \subset  K_f (\ox,\ov).
\end{equation}
Assume now that $f$ is parabolically regular at $\ox$ for  $\ov $. If there exists a $w\in K_f (\ox,\ov)\setminus \dom \d^2 f (\bar x , \ov )$, then \eqref{pri1} clearly holds due to \eqref{pr2}.
Suppose now that  $w \in \dom \d^2 f (\bar x , \ov )$. By Definition~\ref{pre},  there are  sequences $t_k \dn 0$ and $w_k \to w$ for which we have 
\begin{equation*}\label{pri4}
\Delta_{t_k}^2  f (\ox,\ov)(w_k) \to \d^2 f (\bar x,\ov)(w)\quad \mbox{and} \quad
\limsup_{k\to \infty} \frac{\|w_k-w\|}{t_k}<\infty.
\end{equation*}
Since the sequence $z_k:=2[w_k-w]/t_k$ is bounded, we can assume  by passing to a subsequence if necessary that $ z_k \to \oz$ as $k \to \infty $ for some $\oz\in \X$.
Thus we have $w_k = w + \frac{1}{2} t_k z_k$ and 
\begin{eqnarray*}\label{pri5}
 \d^2 f (\bar x,\ov)(w) & = & \lim_{k \to \infty}  \dfrac{f (\ox+t_k w_k)- f(\ox)- t_k\langle \ov , w_k \rangle}{\frac{1}{2}t_k^2} \\
&=& \lim_{k \to \infty}   \dfrac{f (\ox+t_k w + \frac{1}{2} t_k^2 z_k )- f(\ox)- t_k\langle \ov , w \rangle}{\frac{1}{2}t_k^2}  - \la \ov ,z_k  \ra  \\
 &\ge & \liminf_{k \to \infty}  \dfrac{f (\ox+t_k w + \frac{1}{2} t_k^2 z_k )- f(\ox)- t_k \d f(\ox)(w)  }{\frac{1}{2}t_k^2}  - \la \ov ,\oz  \ra \\
&\geq &  \d^2 f(\ox)(w\verl \oz) -\la \ov , \oz \ra.
\end{eqnarray*}
Combining this and   \eqref{pr2} implies that   \eqref{pri1} and \eqref{pri11} hold for all $w\in  \dom \d^2 f (\bar x , \ov )$.
To obtain the opposite implication, assume that \eqref{pri1} holds for all $w\in K_f (\ox,\ov)$.
To prove the parabolic regularity of $f$ at  $\ox$ for $\ov$, let $\d^2 f (\bar x , \ov ) (w)<\infty$, which by \eqref{dss1}  yields  $w\in K_f (\ox,\ov)$. 
Employing now \cite[Proposition~13.64]{rw} results in  
$$
\d^2 f (\bar x , \ov )(w)= \inf_{z\in \X}\big\{\d^2 f(\ox)(w\verl z) -\la z, \ov\ra\big \}  =\liminf_{\substack{
   t\dn 0, \,w'\to w \\
  [w'-w]/t\,\, {\ss\mbox{bounded}}
  }} \Delta_t^2  f(\ox , \ov)(w').
$$ 
The last equality clearly justifies \eqref{prc}, and thus $f$ is parabolically regular at $\ox$ for $\ov$. This completes the proof.
\end{proof}

We next show that the indicator function of the cone of $n\times n$ positive  semidefinite symmetric matrices, denoted by ${\S}^n_+$, 
is parabolic regular. This can be achieved via \cite[Theorem~6.2]{mms2}  using the theory of ${\cal C}^2$-cone reducible sets but below  we give an 
independent proof via Proposition~\ref{parregularity}.
\begin{Example}[parabolic regularity of ${\S}^n_+$]  
{\rm Let ${\S}^n_-$ stand for the cone of $n\times n$ negative semidefinite symmetric matrices. 
For any $A\in \S^n_-$, we are going to show that $f:=\dd_{\S^n_-}$ is parabolic regular at $A$ for any $V\in N_{\S^n_-}(A)$.
Since we have $\S^n_+=-\S^n_-$, this clearly yields the same property for $\S_+^n$. Using the notation in Example~\ref{eig},
we can equivalently write 
\begin{equation}\label{sn}
\S^n_-=\big\{A\in \S^n|\, \lm_1(A)\le 0\big\},
\end{equation}
which in turn implies that $\dd_{\S^n_-}(A)=\dd_{\R_-}(\lm_1(A))$ for any $A\in \S^n$.
If $A$ is negative definite, i.e., $\lm_1(A)<0$, then our claim immediately follows from $N_{\S^n_-}(A)=\{0\}$ for this case. 
Otherwise, we have $\lm_1(A)=0$. Pick $V\in  N_{\S^n_-}(A)$ and conclude from 
\eqref{sn} and the chain rule from convex analysis that $N_{\S^n_-}(A)=\R_+\sub \lm_1(A)$,
which implies that $V=rB$ for some $r\in \R_+$ and $B\in \sub \lm_1(A)$. If $r=0$, we get $V=0$ and 
parabolic regularity of $\dd_{\S^n_-}$  at $A$ for $V$ follows directly from the definition. Assume now $r>0$ and 
pick  $W\in K_f(A,V)$. The latter amounts to 
$$
\la V,W\ra=\d \dd_{\S^n_-}(A)(W)=0\quad \mbox{and }\;\; W\in T_{\S^n_-}(A).
$$
Employing now \cite[Proposition~2.61]{bs} tells us that $\d \lm_1(A)(W)\le 0$. Since $B\in \sub \lm_1(A)$ and $\la B,W\ra=0$, we
arrive at $\d \lm_1(A)(W)= 0$.  We know from  \cite[Example~10.28]{rw} that 
$$
\d\lm_1(A)(W)=\lim_{\substack{
   t\dn 0 \\
  W'\to W
  }}\Delta_t \lm_1(A)(W')\quad \mbox{with}\;\;\Delta_t \lm_1(A)(W'):=\frac{\lm_1(A+tW')-\lm_1(A)}{t}.
$$
Using direct calculations, we conclude  for any $t>0$ and $W'\in \S^n$ that 
\begin{eqnarray*}
\Delta_t^2 \dd_{\S^n_-}(A, V)(W')=\Delta_t^2 \dd_{\R_-}(\lm_1(A), r)(\Delta_t \lm_1(A)(W'))+ r\Delta_t^2 \lm_{1}(A, B)(W'),
\end{eqnarray*}
which in turn results in 
$$
\d^2 \dd_{\S^n_-}(A, V)(W)\ge \d^2\dd_{\R_-}(\lm_1(A), r)(\d \lm_1(A)(W))+ r\d^2 \lm_{1}(A, B)(W).
$$
Since $r>0$, $\lm_1(A)=0$, and $\d \lm_1(A)(W)= 0$, we conclude from \cite[Example~3.4]{mms2} that 
$$
\d^2\dd_{\R_-}(\lm_1(A), r)(\d \lm_1(A)(W))=\dd_{K_{\R_-}(\lm_1(A), r)}(0)=\dd_{\{0\}}(0)=0.
$$ 
Using this together with \eqref{alss} brings us to 
$$
\d^2 \dd_{\S^n_-}(A, V)(W)\ge -2r\la B, WA^{\dagger}W\ra=-2\la V, WA^{\dagger}W\ra.
$$
On the other hand, we conclude from \eqref{pri3} that 
$$
\d^2 \dd_{\S^n_-}(A, V)(W)\le -\sigma_{T^2_{\S^n_-}(A, W)}(V)= -2\la V, WA^{\dagger}W\ra,
$$
where the last equality comes from \cite[page~487]{bs} with $\sigma_{T^2_{\S^n_-}(A, W)}$ standing for the support function of $T^2_{\S^n_-}(A, W)$.  Combining these confirms that 
$$
\d^2 \dd_{\S^n_-}(A, V)(W)= -\sigma_{T^2_{\S^n_-}(A, W)}(V)= -2\la V, WA^{\dagger}W\ra\quad \mbox{for all}\;\; W\in K_{f}(A,V).
$$
This together with Proposition~\ref{parregularity} tells us that ${\S^n_-}$ is parabolic regular at $A$ for $V$.
}
\end{Example}

We are now in a position to establish the main result of this section, which states that parabolically regular functions are always twice epi-differentiable.

\begin{Theorem}[twice epi-differenitability of parabolically regular functions] \label{tedp} 
Let $f:\X \to \oR$ be  finite at $\ox$ and $\ov \in \sub^p f(\ox)$ and  let $f$ be parabolically epi-differentiable at $\ox$ for every  $ w \in K_f(\ox,\ov)$. 
If $f$ is parabolically regular at $\ox$ for  $\ov$,  then it is properly twice epi-differentiable at $\ox$ for $\ov$ with
\begin{equation}\label{pri2}
\d^2 f (\bar x ,\ov)(w)= \left\{\begin{array}{ll}
\min_{z\in \X}\big\{\d^2 f(\ox)(w\verl z) -\la z, \ov\ra\big \}  \quad&\mbox{if }\;\; w\in K_f (\ox,\ov),\\
+\infty &\mbox{otherwise}.
\end{array}
\right.
\end{equation} 
\end{Theorem} 
\begin{proof} It follows from the parabolic epi-differentiability of $f$ at $\ox$ for every  $ w \in K_f(\ox,\ov)$ and Proposition~\ref{dss} that $\dom \d^2 f (\bar x , \ov ) =  K_f (\ox,\ov)$.
This together with \eqref{pri1} and \eqref{pri11} justifies the second subderivative formula \eqref{pri2}. 
 To establish the twice epi-differentiability of $f$ at $\ox$ for $\ov$, we are going to show that \eqref{dedf} holds for all $w\in \X$.
 Pick  $w \in K_f (\ox , \ov)$ and an arbitrary sequence $t_k\dn 0$. 
 Since $f$ is parabolically regular at $\ox$ for  $\ov$, by Proposition~\ref{parregularity}, we find a $\oz\in \X$ such that 
  \begin{equation}\label{pl01}
 \d^2 f (\bar x , \ov ) (w) =  \d^2 f(\ox)(w\verl \oz) -\la \oz, \ov\ra.
 \end{equation}
By the parabolic epi-differentiability  of $f$ at $\ox$ for $w$, we find a sequence $z_k\to \oz$ for which we have  
$$
\d^2 f(\ox)(w\verl \oz)=\lim_{k\to \infty} \dfrac{f(\ox+t_kw+\frac{1}{2}t_k^2 z_k)-f(\ox)-t_k\d f(\ox)(w)}{\frac{1}{2}t_k^2}.
$$
Define $w_k := w + \frac{1}{2} t_k z_k$ for all $k\in \N$. Using this and $w\in K_f (\ox , \ov)$, we obtain  
\begin{eqnarray*}
\Delta_{t_k}^2 f (\ox,\ov)(w_k)  & = &  \dfrac{f (\ox+t_k w_k)- f(\ox)- t_k\langle \ov , w_k \rangle}{\frac{1}{2} t_k^2} \\
&=&  \dfrac{f (\ox+t_k w + \frac{1}{2} t_k^2 z_k )- f(\ox)- t_k\d f(\ox)(w)}{\frac{1}{2}t_k^2}  - \la \ov ,z_k  \ra.
\end{eqnarray*}
This together with \eqref{pl01} results in 
$$
 \lim_{k\to \infty} \Delta_{t_k}^2 f (\ox,\ov)(w_k) =  \d^2 f(\bar x)(w\verl \oz) - \la \ov ,\oz  \ra  = \d^2 f (\ox,\ov)(w),
 $$ 
 which justifies \eqref{dedf} for every $w\in K_f (\ox , \ov)$. Finally, we are going to show the validity of \eqref{dedf} for every 
 $w\notin K_f (\ox , \ov)$. For any such a $w$, we conclude from \eqref{pri2} that $\d^2 f (\bar x ,\ov)(w)=\infty$.
 Pick an arbitrary sequence $t_k\dn 0$ and set  $w_k:=w$ for all $k\in \N$.  Thus we have 
$$
\infty=\d^2 f(\bar x ,\ov)(w)\le \liminf_{k\to \infty} \Delta_{t_k}^2 f(\bar x ,\ov)(w_k)\le\limsup_{t\dn 0} \Delta_{t_k}^2 f(\bar x ,\ov)(w_k)\le \infty=\d^2 f(\bar x ,\ov)(w),
$$
 which again proves \eqref{dedf} for all $w\notin K_f (\ox , \ov)$. This completes the proof of the Theorem.
\end{proof}

The above theorem provides a very important generalization of a similar result obtained recently by the authors and Mordukhovich  in 
\cite[Theorem~3.6]{mms2} in which the twice epi-differentiability of set indicator functions  was established. It is not hard 
to see that  the assumptions of Theorem~\ref{tedp} boils down to those in \cite[Theorem~3.6]{mms2}. To the best of our knowledge, 
the only results related to the twice epi-differentiability of functions, beyond set indicator functions, are 
\cite[Theorem~13.14]{rw} and \cite[Theorem~3.1]{t2} in which this property was proven  for the fully amenable 
and  eigenvalue functions, respectively. We will derive these results in Section~\ref{sect05}
as an immediate consequence of our chain rule for the second subderivative. 

We proceed with an important consequence of Theorem~\ref{tedp} in which 
the  proto-differentiability  of subgradient mappings is established under parabolic regularity. Recall that a set-valued mapping 
$S:\X\tto \Y$ is said to be {\em proto-differentiable} at $\ox$ for $\oy$ with $(\ox,\oy)\in \gph S$ if the set $\gph S$
is geometrically derivable at $(\ox,\oy)$. When this condition holds for the set-valued mapping $S$ at $\ox$ for $\oy$,
we refer to $DS(\ox, \oy)$ as  the {\em proto-derivative} of $S$ at $\ox$ for $\oy$.
The connection between the twice epi-differentiablity 
of a function and the proto-differentiability of its subgradient mapping was observed first by Rockafellar in \cite{r90}
for convex functions and was extended later in \cite{pr} for prox-regular functions. Recall that a function 
$f:\X\to \oR$ is called prox-regular at $\ox$ for $\ov$ if $f$ is finite at $\ox$ and  is locally l.s.c. around $\ox$ with $\ov\in \sub f(\ox)$ and 
there are constant $\ve>0$ and $r\ge  0$ such that for all $x\in \B_{\ve}(\ox)$ with $f(x)\le f(\ox)+ \ve$ we have
\begin{eqnarray*}\label{prox}
f(x)\geq f(u)+\la v,x-u \ra- \rsm \|x-u\|^2\quad\quad\mbox{for all} \quad (u,v)\in (\gph \sub f)\cap  \B_{\ve}(\ox, \ov).
\end{eqnarray*}
Moreover, we say that $f$ is {subdifferentially continuous} at $\ox$ for $\ov$ if $(x_k,v_k)\to
 (\ox,\ov)$ with $v_k\in \sub f(x_k)$, one has $f(x_k)\to f(\ox)$.

\begin{Corollary}[proto-differentiability under parabolic regularity]\label{pdpr}
Let $f:\X \to \oR$ be prox-regular and subdifferentially continuous at $\ox$ for $\ov$ and  let $f$ be parabolically epi-differentiable at $\ox$ for every  $ w \in K_f(\ox,\ov)$. 
If $f$ is parabolically regular at $\ox$ for  $\ov$,  then the following equivalent conditions hold:
\begin{itemize}[noitemsep,topsep=0pt]
\item [\bf{(i)}] the function $f$ is twice epi-differentiable at $\ox$ for $\ov$;
\item [\bf{(ii)}] the subgradient  mapping $\sub f$ is proto-differentiable at $\ox$ for $\ov$.
\end{itemize}
Furthermore, the proto-derivative of the subgradient mapping $\sub f$ at $\ox$ for $\ov$ can be calculated by   
\begin{equation}\label{gdpd}
D(\sub f)(\ox, \ov)(w)= \sub \big(\sm\d^2f(\ox, \ov)\big)(w)\quad \mbox{for all}\;\; w\in \X.
\end{equation}

\end{Corollary} 
\begin{proof}Note  that $\ov\in \sub^p f(\ox)$ since $f$ is prox-regular at $\ox$ for $\ov$. 
Employing now   Theorem~\ref{tedp} gives us (i). The equivalence between (i) and (ii) 
and the validity of \eqref{gdpd} come from \cite[Theorem~13.40]{rw}.
\end{proof}

\section{Variational Properties of Parabolic Subderivatives}\sce \label{sect04}

This section is devoted to second-order analysis of parabolic subderivatives of  extended-real-valued functions that are locally Lipschitz continuous relative to their domains.
We pay special attention to functions that are expressed as a composition of a convex function and a twice differentiable function. 
We begin with the following result that gives us sufficient conditions for finding the domain of the parabolic subderivative. 
 
\begin{Proposition}[properties of  parabolic subderivatives]\label{psd}
Let $f: \X \to \oR$ be finite at $\ox$  and let  $f$ be Lipschitz continuous around $\ox$ relative to its domain with constant $\ell\in \R_+$. 
Assume that $w \in T_{\ss \dom f}(\ox)$ and that $f$ is parabolic epi-differentiable at $\ox$ for  $w$. Then  the following conditions hold:
\begin{itemize}[noitemsep,topsep=1pt]
\item [\bf{(i)}]    $\dom \d^2 f(\bar x)(w\verl \cdot) = T_{\ss \dom f}^2 (\ox ,  w)$;
\item [\bf{(ii)}] $\dom f$ is parabolically derivable at $\ox$ for $w$.
\end{itemize}
\end{Proposition}
\begin{proof} Since $w\in T_{\ss \dom f}(\ox)$, we conclude from Proposition~\ref{dfs} that $\d f(\ox)(w)$ is finite. To prove (i), observe first that by definition, we always have the inclusion 
\begin{equation}\label{domps}
\dom \d^2 f(\bar x)(w\verl \cdot) \subset T_{\ss \dom f}^2 (\ox, w).
\end{equation}
 To obtain the opposite inclusion, take  $z \in T_{\ss \dom f}^2 (\ox,  w)$. This tells us that there exist sequences $t_k\dn 0$ and $z_k\to z$ 
 so that $\ox + t_kw + \frac{1}{2} t_k^2 z_k  \in \dom f $. Since $f$ is parabolically epi-differentiable at $\ox$ for $w$, we have $\dom \d^2 f(\bar x)(w\verl \cdot) \neq\emptyset$.
 Thus there exists a $z_w\in \X$ such that $\d^2 f(\bar x)(w\verl z_w) <\infty$.
 Moreover, corresponding to the sequence $t_k$, we find another sequence $z'_k\to z_w$ such that 
 $$
\d^2 f(\ox)(w\verl z_w)=\lim_{k\to \infty} \dfrac{f(\ox+t_kw+\frac{1}{2}t_k^2 z'_k)-f(\ox)-t_k\d f(\ox)(w)}{\frac{1}{2}t_k^2}.
$$
Since $\d^2 f(\bar x)(w\verl z_w) <\infty$, we can assume without loss of generality that $\ox+t_kw+\frac{1}{2}t_k^2 z'_k\in \dom f$ for all $k\in \N$.
Using these together with the Lipschitz continuity of $f$ around $\ox$ relative to its domain,  we have for all $k$ sufficiently large  that 
\begin{eqnarray*}
\frac{f(\ox+t_k w+\frac{1}{2} t_k^2 z_k)-f(\ox)-t_k \d f(\ox)(w)} {\frac{1}{2}t_k^2}&=&\frac{f(\ox+t_k w+\frac{1}{2}t_k^2 z'_k)-f(\ox)-t_k\d f(\ox)(w)}{\frac{1}{2}t_k^2} \\ 
&&+ \frac{f(\ox+t_k w+\frac{1}{2}t_k ^2 z_k) -f(\ox+t_k w+\frac{1}{2}t_k ^2 z'_k)}{\frac{1}{2}t_k^2}\\
&\leq & \frac{f(\ox+t_k w+\frac{1}{2}t_k^2 z'_k)-f(\ox)-t_k \d f(\ox)(w)}{\sm t_k^2}\\ 
&&+ \ell   \| z_k - z'_{k} \|.
\end{eqnarray*}
Passing to the limit results in the inequality 
\begin{equation}\label{lipps} 
 \d^2 f(\bar x)(w\verl z) \leq \d^2 f(\bar x)(w\verl z_{w}) + \ell \|z - z_{w} \|,
\end{equation}
 which in turn yields $\d^2 f(\bar x)(w\verl z) <\infty$, i.e., $z\in \dom \d^2 f(\bar x)(w\verl \cdot)$. This justifies the opposite inclusion in \eqref{domps}
 and hence proves (i). 
 
 Turning now to (ii), we conclude from \eqref{domps} and the parabolic epi-differentiability of $f$ at $\ox$ for $w$ that 
the second-order tangent set $ T_{\ss \dom f}^2 (\ox, w)$ is nonempty. Moreover, it follows from \cite[Example~13.62(b)]{rw} that 
the parabolic epi-differentiability of $f$ at $\ox$ for $w$ yields the parabolic derivability of $\epi f$ at $(\ox,f(\ox))$ for $(w,\d f(\ox)(w))$. 
The latter clearly enforces the same property for $\dom f$ at $\ox$ for $w$ and hence completes the proof.
\end{proof} 

It is important to notice the parabolic epi-differentiability of $f$ in Proposition~\ref{psd} is essential 
to ensure that condition (i) therein, namely the characterization of the domain of the parabolic subderivative, is satisfied.
Indeed, as mentioned in the proof of this proposition, inclusion \eqref{domps} always holds. If the latter condition fails, this 
inclusion can be strict. For example, the function $f$ from Example~\ref{ex01} is not parabolic epi-differentiable at $\ox=(0,0)$
for any vector $w=(w_1,w_2)\in  K_f (\ox,\ov)$ with $w_1\neq 0$ since $\dom  \d^2 f(\bar x)(w\verl \cdot)=\emptyset$. On the other hand, 
we have $\dom f=\R^2$ and thus $T_{\ss \dom f}^2 (\ox, w)=\R^2$ for any such a vector $w\in K_f (\ox,\ov)$, and so  condition (i) in Proposition~\ref{psd}  fails. 

Given a function $f:\X\to \oR$ finite at $\ox$,  in the rest of this paper, we mainly focus on the case when  this function has a 
representation of the form 
\begin{equation}\label{CF}
f (x)= (g \circ F)(x)\quad  \mbox{for all}\;\;  x\in {\cal O},
\end{equation}
where ${\cal O}$ is a neighborhood of $\ox$ and where the functions $F$ and $g$ are satisfying the following conditions:
\begin{itemize}[noitemsep,topsep=2pt]
\item $F:\X \to \Y$ is twice differentiable at $\ox$;
\item $g: \Y \to \oR$ is  proper, l.s.c.,\ convex,  and  Lipschitz continuous around $F(\ox)$ relative to its domain with constant $\ell\in \R_+$.
\end{itemize}
It is not hard to see that the imposed assumptions on  $g$ from representation \eqref{CF} implies that 
$\dom g$ is locally closed around $F(\ox)$, namely for some $\ve>0$ the set $(\dom g)\cap \B_\ve(F(\ox))$ is closed. 
Taking the neighborhood ${\cal O}$ from  \eqref{CF},  we obtain  
\begin{equation}\label{CS}
(\dom f)\cap {\cal O}= \big\{ x \in {\cal O} | \: F(x) \in \dom g\big \}.
\end{equation}
 It has been well understood that  the second-order variational analysis of the composite 
form \eqref{CF} requires a certain qualification condition. The following definition provides the one we utilize in this paper.
\begin{Definition}[metric subregularity constraint qualification] \label{mscq} Assume that the function $f:\X\to \oR$ has representation \eqref{CF}
around $\ox\in \dom f$. We say that the metric subregularity constraint qualification holds  for the constraint set \eqref{CS} at $\ox$ if 
 there exist a constants $\kappa \in \R_+$ and a neighborhood $U$ of $\ox$  such that
\begin{equation}\label{mscq1}
d(x , \dom f) \leq \kappa \: d(F(x) , \dom g) \quad \mbox{for all}\;\; x \in U.
\end{equation}
\end{Definition}

By definition, the metric subregularity constraint qualification 
for the constraint set \eqref{CS} at $\ox$ amounts to the  metric subregularity of the mapping $x\mapsto F(x)- \dom g $ at $(\ox,0)$.
 The more traditional and well known constraint  qualification for \eqref{CF} is 
\begin{equation*}
\partial^\infty g\big(F(\ox)\big)\cap\ker\nabla F(\ox)^*=\{0\},
\end{equation*}
where $\partial^\infty\ph(\ox)$ stands for the {\em singular subdifferential} of $\ph\colon\R^n\to\oR$ at $\ox\in\dom\ph$ defined by
\begin{equation*}\label{ss}
\partial^\infty\ph(\ox):=\big\{v\in\R^n\big|\;(v,0)\in N_{{\scriptsize\epi\ph}}\big(\ox,\ph(\ox)\big)\big\}.
\end{equation*}
By the coderivative criterion from \cite[Theorem~3.3]{mor18}, the latter  constraint qualification amounts to the metric regularity of the mapping $(x,\al)\mapsto (F(x),\al)-\epi g$ around 
$\big((\ox,g(F(\ox))),(0,0)\big)$. Since $g$ is convex, by \cite[Proposition~1.25]{mor18}, the aforementioned constraint qualification can be equivalently written as 
\begin{equation}\label{gf07}
N_{\ss \dom g} \big(F(\ox)\big)\cap\ker\nabla F(\ox)^*=\{0\},
\end{equation}
which by the coderivative criterion from \cite[Theorem~3.3]{mor18} is equivalent to the metric regularity of the mapping $x\mapsto F(x)- \dom g $ around $(\ox,0)$. 
Thus, for the composite form \eqref{CF}, the metric regularity of $(x,\al)\mapsto (F(x),\al)-\epi g$ around $\big((\ox,g(F(\ox))),(0,0)\big)$ amounts to that of 
$x\mapsto F(x)- \dom g $ around $(\ox,0)$. This may not be true if the metric regularity is replaced with the metric subregularity; see \cite[Proposition~ 3.1]{mms1} and \cite[Example~3.3]{mms1} for more details
and discussions about these conditions. So, it is worth reiterating that instead of using the metric subregularity of the {\em epigraphical} mapping, 
our second-order variational analysis  will be carried out under 
 a weaker (and simpler) metric subregularity of the {\em domain} mapping.
 
 As observed recently in \cite{mms1}, \eqref{mscq1} suffices to conduct first- and second-order 
variational analysis of \eqref{CF} when the convex function $g$ therein is merely  piecewise linear-quadratic. In what follows, we will show using a different approach that 
such results can be achieved for \eqref{CF} as well. We continue  our analysis by recalling the following first- and second-order chain rules, obtained recently in \cite{mms1,mms2}.

\begin{Proposition}[first- and second-order chain rules]\label{fsch} Let  $f:\X\to \oR$ have the composite representation \eqref{CF} at $\ox\in \dom f$ and $\ov\in \sub f(\ox)$ and let 
the metric subregularity constraint qualification hold  for the constraint set \eqref{CS} at $\ox$. Then the following hold:
\begin{itemize}[noitemsep,topsep=1pt]
\item [\bf{(i)}] for any $w\in \X$, the following subderivative  chain rule for  $f$ at $\ox$ holds:
$$
\d f(\ox)(w)=\d g(F(\ox))( \nabla F(\ox) w );
$$
\item [\bf{(ii)}] we have the chain rules
$$
\sub^p f(\ox) = \sub f(\ox) = \nabla F(\ox)^* \sub g(F(\ox))\quad \mbox{and}\quad T_{\ss \dom f}(\ox) =\big\{w\in \X |\; \nabla F(\ox) w \in T_{\ss \dom g}(F(\ox)) \big\}.
$$
\end{itemize}
If, in addition, $w\in T_{\ss \dom f}(\ox)$  and the function $g$ from \eqref{CF} is parabolically  epi-differentiable  at $F(\ox)$ for $ \nabla F(\ox)w$,  then we have
\begin{equation}\label{chrs}
z\in T^2_{\ss \dom f}(\ox,  w) \iff \nabla F(\ox)z +\nabla^2 F(\bar x)(w,w) \in T^2_{\ss \dom g}\big (F(\bar x), \nabla F(\ox)w \big).
\end{equation}  
Moreover, $\dom f$ is parabolically derivable at $\ox$ for $w$.
\end{Proposition}
\begin{proof} The subderivative chain rule  in (i) was established recently in \cite[Theorem~ 3.4]{mms1}.
The subdifferential chain rule in (ii) was taken from  \cite[Theorem~ 3.6]{mms1}. 
As mentioned in Section~\ref{sect02}, the inclusion $\sub^p f(\ox)\subset \sub f(\ox)$ always holds. The opposite inclusion can be 
justified using the aforementioned subdifferential chain rule and the convexity of $g$; see  \cite[Theorem~4.4]{mms1} for a similar result.
The chain rule for the tangent cone to $\dom f$ at $\ox$ results from Proposition~\ref{dfs} and the subderivative 
 chain rule for $f$ at $\ox$ in (i). If in addition $w\in T_{\ss \dom f}(\ox)$ and $g$ is parabolically  epi-differentiable  at $F(\ox)$ for $  \nabla F(\ox)w$, then it follows from Proposition~\ref{psd} that $\dom g$ is parabolically derivable at 
 $F(\ox)$ for $ \nabla F(\ox)w$.  Appealing now to \cite[Theorem~4.5]{mms2} implies that 
 $\dom f$ is parabolically derivable at $\ox$ for $w$. Finally, the chain  rule \eqref{chrs}
 was taken from \cite[Theorem~4.5]{mms2}. This completes the proof.
\end{proof}

We continue by establishing a chain rule for the parabolic subderivative, which is important for  our developments in the next section.

\begin{Theorem}[chain rule for parabolic subderivatives]\label{sochpar}
Let  $f:\X\to \oR$ have the composite representation \eqref{CF} at $\ox\in \dom f$ and $w\in T_{\ss \dom f}(\ox)$ and let 
the metric subregularity constraint qualification hold  for the constraint set \eqref{CS} at $\ox$. 
Assume that  the function $g$ from \eqref{CF} is parabolically  epi-differentiable  at $F(\ox)$ for  $ \nabla F(\ox)w$. Then the following conditions are satisfied:
\begin{itemize}[noitemsep,topsep=0pt]
\item [\bf{(i)}] for any $z\in \X$ we have 
\begin{equation}\label{sochpar1}
\d^2 f(\bar x)(w\verl z) = \d^2 g(F(\ox))( \nabla F(\ox) w \verlm \nabla F(\ox) z + \nabla^2 F(\ox)(w,w));
\end{equation}  
\item [\bf{(ii)}]  the domain of the parabolic subderivative of $f$ at $\ox$ for $w$ is given by
$$
\dom \d^2 f(\bar x)(w\verl \cdot) = T_{\ss \dom f}^2 (\ox ,  w);
$$
\item [\bf{(iii)}]  $f$ is parabolically epi-differentiable at $\ox$ for $w $.
\end{itemize}
\end{Theorem}
\begin{proof} Pick $z\in \X$ and set $u:= \nabla F(\ox) z + \nabla^2 F(\ox)(w,w)$. We prove (i)-(iii) in a parallel way. Assume that  $z\notin T_{\ss \dom f}^2 (\ox ,  w)$. 
As mentioned in the proof of Proposition~\ref{psd}, inclusion \eqref{domps}
always holds. This implies that $\d^2 f(\bar x)(w\verl z)=\infty$. On the other hand,  by  \eqref{chrs} we get $u \notin T^2_{\ss \dom g}\big (F(\bar x), \nabla F(\ox)w \big)$.
Employing  Proposition~\ref{psd}(i) for the function $g$ and $\nabla F(\ox)w\in T_{\ss \dom g}(F(\ox))$  gives us  
\begin{equation}\label{domgp}
\dom  \d^2 g(F(\ox))( \nabla F(\ox) w \verlm\cdot)= T_{\ss \dom g}^2 (F(\ox) ,  \nabla F(\ox)w).
\end{equation}
Combining these tells us that 
$
\d^2 g(F(\ox))( \nabla F(\ox) w \verlm u)=\infty,
$
which in turn justifies \eqref{sochpar1} for every $z\notin T_{\ss \dom f}^2 (\ox ,  w)$. Consider an arbitrary sequence  $t_k\dn 0$ and set $z_k:=z$ for all $k\in \N$. Then 
we have 
\begin{eqnarray*}
\d^2 f(\bar x)(w\verl z) &\le & \liminf_{k\to \infty}  \dfrac{f(\ox+t_kw+\frac{1}{2}t_k^2 z_k)-f(\ox)-t_k\d f(\ox)(w)}{\frac{1}{2}t_k^2}\\
&\le&  \limsup_{k\to \infty}  \dfrac{f(\ox+t_kw+\frac{1}{2}t_k^2 z_k)-f(\ox)-t_k\d f(\ox)(w)}{\frac{1}{2}t_k^2}\\
&\le&\infty=\d^2 f(\bar x)(w\verl z),
\end{eqnarray*} 
which in turn justifies \eqref{pepi} for all $z\notin T_{\ss \dom f}^2 (\ox ,  w)$. 
 
 Since $g$ is parabolically  epi-differentiable  at $F(\ox)$ for  $ \nabla F(\ox)w$,  Proposition~\ref{psd}(ii) 
 tells us that $\dom g$ is parabolically derivable at $F(\ox)$ for  $ \nabla F(\ox)w$. We conclude from Proposition~\ref{fsch} that $\dom f$ is parabolically derivable at $\ox$ for $w$. 
 In particular, we have 
 \begin{equation}\label{pdfps}
 T_{\ss \dom f}^2 (\ox,  w)\neq \emptyset.
 \end{equation}
 Pick now  $z \in T_{\ss \dom f}^2 (\ox,  w)$ and then consider an arbitrary sequence  $t_k\dn 0$. 
Thus, by definition, for the aforementioned sequence $t_k$, we find a sequence $z_k\to z$ as $k\to \infty$ such that 
\begin{equation}\label{seqz}
x_k:=\ox+t_kw+\frac{1}{2}t_k^2 z_k\in \dom f\quad \mbox{for all}\;\;k\in \N.
\end{equation}
Moreover, since $g$ is parabolically epi-differentiable at $F(\ox)$ for $\nabla F(\ox) w$,  we find a sequence 
$u_k\to u $ such that 
\begin{equation}\label{peg}
\d^2 g(F(\ox))( \nabla F(\ox) w \verlm u)= \lim_{k\to \infty}  \dfrac{g(F(\ox)+t_k\nabla F(\ox)w+\frac{1}{2}t_k^2 u_k)-g(F(\ox))-t_k\d g(F(\ox))(\nabla F(\ox)w)}{\frac{1}{2}t_k^2}.
\end{equation}
It follows  from  \eqref{chrs} that  $u\in T^2_{\ss \dom g}\big (F(\bar x), \nabla F(\ox)w \big)$. Combining this with  \eqref{domgp} tells us that    $\d^2 g(F(\ox))( \nabla F(\ox) w \verlm u)<\infty$.
This implies that $y_k:=F(\ox)+t_k\nabla F(\ox)w+\frac{1}{2}t_k^2 u_k\in \dom g$ for all $k$ sufficiently large. 
Remember that $g$ is Lipschitz continuous around $F(\ox)$ relative to its domain with constant $\ell$. Using this together with  Proposition~\ref{fsch}(i),  \eqref{seqz}, and \eqref{peg}, 
we obtain 
\begin{eqnarray}
\d^2 f(\bar x)(w\verl z) &\le & \liminf_{k\to \infty}  \dfrac{f(\ox+t_kw+\frac{1}{2}t_k^2 z_k)-f(\ox)-t_k\d f(\ox)(w)}{\frac{1}{2}t_k^2}\nonumber\\
&\le&  \limsup_{k\to \infty}  \dfrac{f(\ox+t_kw+\frac{1}{2}t_k^2 z_k)-f(\ox)-t_k\d f(\ox)(w)}{\frac{1}{2}t_k^2} \nonumber\\
&=& \limsup_{k\to \infty}  \dfrac{g(F(x_k))-g(F(\ox))-t_k\d g(F(\ox))(\nabla F(\ox)w)}{\frac{1}{2}t_k^2} \nonumber\\
&\le &  \lim_{k\to \infty}  \dfrac{g(y_k)-g(F(\ox))-t_k\d g(F(\ox))(\nabla F(\ox)w)}{\frac{1}{2}t_k^2} +\limsup_{k\to \infty}\dfrac{g(F(x_k))-g(y_k)}{\frac{1}{2}t_k^2} \nonumber\\
&\le & \d^2 g(F(\ox))( \nabla F(\ox) w \verlm u) +\limsup_{k\to \infty}\ell \|\nabla F(\ox)z_k+\nabla^2F(\ox)(w,w)-u_k+\frac{o(t_k^2)}{t_k^2} \| \nonumber\\
&=& \d^2 g(F(\ox))( \nabla F(\ox) w \verlm u)\label{psgf}.
\end{eqnarray}
On the other hand, it is not hard to see that for any $z\in \X$, we always have 
$$
 \d^2 g(F(\ox))( \nabla F(\ox) w \verlm u) \le \d^2 f(\bar x)(w\verl z).
$$
Combining this and \eqref{psgf} implies that 
$$
 \d^2 f(\bar x)(w\verl z)=\d^2 g(F(\ox))( \nabla F(\ox) w \verlm u)
$$
and that 
$$
\d^2 f(\bar x)(w\verl z) = \lim_{k\to \infty}  \dfrac{f(\ox+t_kw+\frac{1}{2}t_k^2 z_k)-f(\ox)-t_k\d f(\ox)(w)}{\frac{1}{2}t_k^2}.
$$
These prove \eqref{sochpar1}  and \eqref{pepi} for any $z\in T_{\ss \dom f}^2 (\ox, w)$, respectively, and hence  we finish the proof of (i).

Next, we are going to verify (ii). We already know that  inclusion \eqref{domps} always holds.
To derive the opposite inclusion, pick $z\in T_{\ss \dom f}^2 (\ox, w)$, which amounts to 
$u\in T^2_{\ss \dom g}\big (F(\bar x), \nabla F(\ox)w \big)$  due to \eqref{chrs}. 
 By (i) and \eqref{domgp}, we obtain 
 $$
 \d^2 f(\bar x)(w\verl z)=\d^2 g(F(\ox))( \nabla F(\ox) w \verlm u)<\infty.
 $$
 This tells us that $z\in \dom \d^2 f(\bar x)(w\verl\cdot)$ and hence completes the proof of (ii).
 
 Finally, to justify (iii), we require to prove the fulfillment of \eqref{pepi} for all $z\in \X$ and
 to show that $\dom  \d^2 f(\bar x)(w\verl \cdot)\neq\emptyset$. The former was proven above and so we proceed with the proof of the latter. 
 This, indeed, follows from \eqref{pdfps} and the characterization of $\dom  \d^2 f(\bar x)(w\verl \cdot)$, achieved in (ii), and  thus completes the proof.
\end{proof}

It is worth mentioning that a chain rule for parabolic subderivatives for the composite form \eqref{CF}
was achieved in \cite[Exercise~13.63]{rw} and \cite[Proposition~3.42]{bs}  when  $g$ is merely a proper l.s.c.\ function and  the basic constraint
qualification \eqref{gf07} is satisfied. Replacing the latter condition with the significantly weaker condition \eqref{mscq1}, we can achieve a similar result if we assume further that   
 $g$ is convex and locally Lipschitz continuous relative to its domain. Another important difference between Theorem~\ref{sochpar} and those mentioned above is that 
 the chain rule \eqref{sochpar1} obtained in \cite{bs,rw} does not require the parabolic epi-differentiability of $g$. Indeed, the usage of 
 the basic constraint qualification \eqref{gf07} in \cite{bs,rw} allows to achieve \eqref{sochpar1} via a chain rule for the epigraphs of $f$ and $g$ 
 similar to the one in \eqref{chrs}, which is not conceivable   under \eqref{mscq1}. These extra assumptions on $g$ automatically fulfill in many important 
 composite and constrained optimization problems and so  do not seem to be restrictive in our developments.
 
 We continue by establishing two important properties for parabolic subderivatives that play crucial roles in our developments in the next 
 section. One notable difference between the following results and those obtained in Proposition~\ref{psd} and Theorem~\ref{sochpar}
 is that we require the parabolic subderivative be proper. This can be achieved if the parabolic subderivative is  bounded below.
 In general, we may not be able to guarantee this. It turns out, however,  that if the vector $w$ in the pervious results is taken from 
 the critical cone to the function in question, which is a subset of the tangent cone to the domain of that function, this can be accomplished 
 via \eqref{pri3}. Since we only conduct our analysis in the next section over the critical cone, this will provide no harm.
 Below, we first show that the parabolic subderivative of an extended-real-valued function, which is locally Lipschitz continuous relative to its domain,  
 is Lipschitz continuous relative to its domain. 
 
 \begin{Proposition}[Lipschitz continuity of  of  parabolic subderivatives]\label{lcps}
Let $\psi: \X \to \oR$ be finite at $\ox$  and $\ov\in \sub^p \psi(\ox)$, and let  $\psi$ be Lipschitz continuous around $\ox$ relative to its domain with constant $\ell\in \R_+$. 
Assume that $w \in K_\psi(\ox, \ov)$ and that $\psi$ is parabolically epi-differentiable at $\ox$ for  $w$. Then  
the parabolic subderivative  $\d^2 \psi(\bar x)(w\verl \cdot)$ is  proper, l.s.c., and Lipschitz continuous relative to its domain with constant $\ell$.
\end{Proposition}
\begin{proof}   Since $\psi$ is parabolically epi-differentiable at $\ox$ for  $w$, we get  $\dom \d^2 \psi(\bar x)(w\verl \cdot) \neq\emptyset$.  
 Let  $z\in \dom \d^2 \psi(\bar x)(w\verl \cdot)$. By Proposition~\ref{dss}, we find  $r\in \R_+$ such that 
\begin{equation}\label{psips}
-r \|w\|^2 \leq  \d^2 \psi (\bar x \verl \ov ) (w) \leq  \d^2 \psi(\ox)(w\verl z) -\la z, \ov\ra.
\end{equation}
 This tells us that $\d^2 \psi(\ox)(w\verl z)$ is finite for every $z\in \dom \d^2 \psi(\bar x)(w\verl \cdot)$ and 
 thus  the parabolic subderivative $\d^2 \psi(\ox)(w\verl \cdot)$ is proper. Pick now $z_i\in \dom \d^2 \psi(\bar x)(w\verl \cdot)$ for $i=1,2$. By Proposition~\ref{psd}(i), we have 
$z_i \in T_{\ss \dom \psi}^2 (\ox , w)$ for $i=1,2$. Letting $z:=z_1$ and $z_w:=z_2$ in \eqref{lipps} results in 
$$
  \d^2 \psi(\bar x)(w\verl z_1) \le  \d^2 \psi(\bar x)(w\verl z_2)  + \ell \|z_1 - z_2\|.
$$ 
Similarly, we can let $z:=z_2$ and $z_w:=z_1$ in \eqref{lipps} and obtain  
$$
 \d^2 \psi(\bar x)(w\verl z_2) \le  \d^2 \psi(\bar x)(w\verl z_1)  + \ell \|z_1 - z_2\|.
$$
Combining these implies that the parabolic subderivative is Lipschitz continuous relative to its domain. By \cite[Proposition~13.64]{rw}, the parabolic subderivative  is always  an l.s.c.\  function, which  completes the proof.
\end{proof} 

We end this section by obtaining an exact formula for the conjugate function of the parabolic subderivative 
of a convex function. 

\begin{Proposition}[conjugate  of parabolic subderivatives]\label{dpd}
Let $\psi:\X\to \oR$ be an l.s.c.\ convex function  with $\psi(\ox)$ finite, $\ov\in \sub \psi(\ox), and $ $w\in K_\psi(\ox,\ov)$. Define the function $\ph$ by  $\ph(z):= \d^2 \psi(\ox)( w\verl z)$ for any $z\in \X$.
If $\psi$ is parabolically epi-differentiable at $\ox$ for $w$ and parabolically regular at $\ox$ for every $v\in \sub \psi(\ox)$, 
then $\ph$ is a proper, l.s.c., and  convex function and its conjugate function is given by 
\begin{equation}\label{conj}
\ph^*(v)=\begin{cases}
- \d^2 \psi(\ox , v) (w)&\mbox{if}\;\;  v\in {\cal A}(\ox,w),\\
\infty& \mbox{otherwise},
\end{cases}
\end{equation}
where ${\cal A}(\ox,w):=\{v\in \sub \psi(\ox)|\, \d \psi(\ox)(w)=\la v,w\ra\}$.
\end{Proposition} 
\begin{proof}  It follows from \cite[Proposition~13.64]{rw} that $\ph$  is   l.s.c.
 Using similar arguments 
as the beginning of the proof of Proposition~\ref{lcps} together with \eqref{psips} tells us that 
$\ph$ is proper.  Also we deduce from \cite[Example~13.62]{rw} that 
$$
\epi \ph=T^2_{\ss \epi \psi}\big((\ox, \psi(\ox)), (w, \d \psi(\ox)(w))\big),
$$
and thus the parabolic epi-differentiability of $\psi$ at $\ox$ for $w$ amounts to the parabolic derivability of $\epi \psi $ at $(\ox, \psi(\ox))$ for  $(w, \d \psi(\ox)(w))$.
The latter combined with the convexity of $\psi$ tells us that $\epi \ph$ is a convex set in $\X\times \R$ and so $\ph$ is convex. 

To verify \eqref{conj}, pick $v\in {\cal A}(\ox,w)$. This  yields $v\in \sub \psi(\ox)=\sub^p\psi(\ox)$ and  $w\in K_\psi(\ox,v)$, namely the critical cone of $\psi$ at $(\ox,v)$.
 Using  Proposition~\ref{parregularity}
and parabolic regularity of $\psi$ at $\ox$ for $v$ implies that 
$$
\d^2 \psi(\ox , v) (w)=\inf_{z\in \X}\big\{\d^2 \psi(\ox)( w\verl z) -\la z,v\ra\big\}=-\ph^*(v),
$$
which clearly proves \eqref{conj} in this case. Assume now that $v\notin {\cal A}(\ox,w)$. This means either $v\notin \sub \psi(\ox)$ or $\d \psi(\ox)(w)\neq \la v,w\ra$.
Define the parabolic difference quotients for $\psi$ at $\ox$ for $w$ by 
$$
\vartheta_t( z)=\dfrac{\psi(\ox+tw+\frac{1}{2}t^2 z)-\psi(\ox)-t\d \psi(\ox)(w)}{\frac{1}{2}t^2}, \quad z\in \X, \;\;t>0.
$$
It is not hard to see that  $\vartheta_t$ are proper, convex, and  
$$
\vartheta_t^*(v)=\dfrac{\psi(\ox)+\psi^*(v)-\la v,\ox\ra}{\frac{1}{2}t^2}+ \frac{\d\psi(\ox)(w)-\la v,w\ra}{\frac{1}{2}t}, \quad v\in \X.
$$
Remember  that by \cite[Definition~13.59]{rw} the parabolic epi-differentiability of $\psi$ at $\ox$ for $w$ amounts to 
the sets  $\epi \vartheta_t$ converging to $\epi \ph$ as $t\dn 0$ and that  the functions $\vartheta_t$ and $\ph$ are proper, l.s.c.\! and convex. 
Appealing to \cite[Theorem~11.34]{rw} tells us that the former is equivalent to the sets $\epi \vartheta_t^*$ converging to 
$\epi \ph^*$ as $t\dn 0$. This, in particular, means that  for any $v\notin {\cal A}(\ox,w)$ and any sequence $t_k\dn 0$, there exists a sequence  $v_k\to v$ such that 
$$
\ph^*(v)=\lim_{k\to \infty} \vartheta_{t_k}^*(v_k).
$$
If $v\notin \sub \psi(\ox)$, then we have 
$$
\psi(\ox)+\psi^*(v)-\la v,\ox\ra>0.
$$
Since $\psi^*$ is l.s.c.,\ we get
$$
\liminf_{k\to \infty} \dfrac{\psi(\ox)+\psi^*(v_k)-\la v_k,\ox\ra}{\frac{1}{2}t_k}+ \frac{\d\psi(\ox)(w)-\la v_k,w\ra}{\frac{1}{2}} \ge \infty,
$$
which in turn confirms that 
$$
\ph^*(v)=\lim_{k\to \infty} \vartheta_{t_k}^*(v_k)= \lim_{k\to \infty}\frac{1}{t_k}\big(\dfrac{\psi(\ox)+\psi^*(v_k)-\la v_k,\ox\ra}{\frac{1}{2}t_k}+ \frac{\d\psi(\ox)(w)-\la v_k,w\ra}{\frac{1}{2}}\big)= \infty.
$$
If $v\in \sub \psi(\ox)$ but $\d \psi(\ox)(w)\neq \la v,w\ra$, we obtain $\la v,w\ra<\d \psi(\ox)(w) $. Since we always have 
$$
\psi(\ox)+\psi^*(v_k)-\la v_k,\ox\ra\ge 0\quad \mbox{ for all}\;\; k\in \N,
$$ we arrive at 
$$
\ph^*(v)=\lim_{k\to \infty} \vartheta_{t_k}^*(v_k)\ge \lim_{k\to \infty} \frac{\d\psi(\ox)(w)-\la v_k,w\ra}{\frac{1}{2}t_k}=\infty,
$$
which justifies \eqref{conj} when $v\notin {\cal A}(\ox,w)$ and hence finishes the proof.
\end{proof}

Proposition~\ref{dpd} was first established using a different method in \cite[Proposition~3.5]{r88} for convex piecewise linear-quadratic functions.
 It was extended in \cite[Theorem~3.1]{c91} for any convex functions 
under a restrictive assumption. Indeed, this result demands  that the second subderivative  be the same as the second-order directional derivative.
Although this condition holds for convex piecewise linear-quadratic functions, it fails for many important functions occurring in constrained and composite 
optimization problems including the set indicator functions and eigenvalue functions. As discussed below, however, our assumptions
are satisfied for all these examples. 
\begin{Example}\label{ex04}
{\rm Suppose that  $g:\Y\to \oR$  is an l.s.c.\ convex function and $\oz\in \Y$.
\begin{itemize}[noitemsep,topsep=0pt]
\item [\bf{(a)}] If $\Y=\R^m$,   $g$  is convex piecewise linear-quadratic (Example~\ref{plqf}),  and $\oz\in \dom g$, then it follows from  Example~\ref{plqf} and 
 \cite[Exercise~13.61]{rw} that $g$ is parabolically regular at $\oz$ for every $y\in \sub g(\oz)$ and parabolically epi-differentiable at $\oz$ for every $w\in \dom \d g(\oz)$, respectively, 
 and thus all the assumptions of Proposition~\ref{dpd} are satisfied for this function.

\item [\bf{(b)}] If $\Y=\S^m$,   $g$  is either the maximum eigenvalue function $\lm_{\max}$ from \eqref{maxeig} or the function $\sigma_i$ from \eqref{fsig},  and $A\in \S^n$, 
then by Example~\ref{eig} $g$ is  parabolically regular at $A$ for every $V\in \sub g(A)$. Moreover, we deduce from \cite[Proposition~2.2]{t1} that 
$g$ is parabolically epi-differentiable at $A$ for every $W\in \S^n$ and thus all the assumptions of Proposition~\ref{dpd} are satisfied for these functions.
\item [\bf{(c)}] If $g=\dd_C$ and $\oz\in C$, where $C$ is a closed  convex set in $\Y$ that is parabolically derivable at $\oz$ for every $w\in T_C(\oz)$
and parabolically regular at $\oz$ for every $v\in N_C(\oz)$, then $g$ satisfies the assumptions imposed in  Proposition~\ref{dpd}.
This example of $g$ was recently explored in detail in \cite{mms2} and encompasses important sets appearing in constrained optimization problems
such as polyhedral convex sets, the second-order cone, and the cone of positive semidefinite symmetric matrices.
 \item[\bf{(d)}] Assume that  $g$ is differentiable  at $\oz$ and that  there exists 
 a continuous  function $h : \Y \to \R$, which is positively homogeneous of degree $2$,  such that 
\begin{equation*}
g(z)=g(\oz)+ \la \nabla g(\oz) , z - \oz \ra +  \sm h(z-\oz) + o(\| z - \oz \|^2 ).
\end{equation*}
Such a function $g$ is called twice semidifferentiable (cf. \cite[Example~13.7]{rw}) and often appears 
in the augmented Lagrangian function associated with \eqref{comp}; see \cite[Section~8]{mms2} for more detail.
This second-order expansion clearly justifies the  parabolic epi-differentiability of $g$ at $\oz$ for every  $w\in \Y$. Moreover,  one has 
\begin{equation*}\label{dissquar1}
 \d^2 g(\oz, \nabla g(\oz))(w)=h(w) = \d^2 g(\oz) ( w \verl u) - \la \nabla g(\oz) , u \ra\;\; \mbox{for all}\;\;u,w\in \Y,
\end{equation*}
which in turn shows that $g$ is  parabolically regular at $\oz$ for $\nabla g(\oz)$ due to Proposition \ref{parregularity}.
\end{itemize}

}
\end{Example}

It is important to mention that the restrictive assumption on the second subderivative, used in \cite[Theorem~3.1]{c91}, does not hold for cases (b)-(d)
in Example~\ref{ex04}. 
\vspace{-0.1 in}
 
 \section{ A Chain Rule for Parabolically Regular Functions}\sce \label{sect05}
Our main objective in this section is to derive a chain rule for the parabolic regularity of the composite representation \eqref{CF}.
 This opens the door to obtain a chain rule for the second subderivative, and, more importantly, 
 allows us to establish the twice epi-differentiability of the latter composite form. 

Taking into account representation \eqref{CF} and picking a subgradient $\ov \in \sub f(\ox)$, we 
 define the set of {\em Lagrangian multipliers} associated with $(\ox,\ov)$
 by 
 $$ 
 \Lambda(\ox , \ov) = \big\{y \in \Y |\; \nabla F(\ox)^* y = \ov , \; y \in \sub g(F(\ox)) \big \}.
 $$
 
 In what follows, we say that a function $f:\X\to \oR$ with $(\ox,\ov)\in \gph \sub f$ and having the composite representation \eqref{CF} at $\ox$ satisfies 
 the basic assumptions at $(\ox,\ov)$ if in addition the following conditions fulfill: 
\begin{itemize}[noitemsep,topsep=2pt]
\item [\bf{(H1)}] the metric subregularity constraint qualification holds  for the constraint set \eqref{CS} at $\ox$;
\item [\bf{(H2)}] for any $y\in \Lambda(\ox,\ov)$,  the function $g$ from \eqref{CF} is parabolically  epi-differentiable  at $F(\ox)$ for every $ u \in K_{ g} (F(\ox),y)$;
\item [\bf{(H3)}] for any $y\in \Lambda(\ox,\ov)$,  the function $g$ is parabolically regular at $F(\ox)$ for $y$.
\end{itemize}

We begin with the following result in which 
we collect lower and upper estimates for the second subderivative of $f$ taken from \eqref{CF}.

\begin{Proposition}[properities of second subderivatives for composite functions]\label{lues}
Let $f:\X\to \oR$ have the composite representation \eqref{CF} at $\ox\in \dom f$,  $\ov\in \sub f(\ox)$, and 
let the basic assumptions {\rm(H1)} and {\rm(H2)} hold for $f$ at $(\ox,\ov)$. Then the second subderivative $\d^2 f(\ox , \ov)$ is a proper l.s.c. function with 
\begin{equation}\label{domss}
\dom \d^2 f(\ox , \ov) =K_f(\ox,\ov).
\end{equation}
Moreover, for every $w\in \X$ we have the lower estimate 
\begin{equation}\label{less}
\d^2 f(\ox , \ov) (w)\ge \sup_{y\in \Lambda(\ox,\ov)}\big\{ \la y , \nabla^2 F(\ox) (w,w) \ra + \d^2 g(F(\ox) , y) ( \nabla F(\ox) w )\big\},
\end{equation}
while for every $w\in K_f(\ox,\ov)$ we obtain the upper estimate
\begin{eqnarray}\label{uess}
\d^2 f(\ox , \ov) (w)\le  \inf_{z \in \X} \big\{    -\la z, \ov\ra +\d^2 g(F(\ox))( \nabla F(\ox) w \verlm \nabla F(\ox) z + \nabla^2 F(\ox)(w,w))\big \} <\infty.
\end{eqnarray}
\end{Proposition}
\begin{proof} By Proposition~\ref{fsch}(ii), we have $\sub^p f(\ox)=\sub f(\ox)$. Appealing now to Propositions~\ref{ssp}(iii) and \ref{dss} confirms, respectively,  that 
$\d^2 f(\ox , \ov)$ is a proper l.s.c. function and that \eqref{domss} holds. The lower estimate \eqref{less} can be justified as \cite[Theorem~13.14]{rw} in which 
this estimate was derived under condition \eqref{gf07}.  To obtain \eqref{uess}, observe first that the basic assumption (H1) yields
\begin{equation}\label{cri2}
w\in K_f(\ox,\ov) \iff \nabla F(\ox)w\in K_g(F(\ox),y)
\end{equation}
for every $y\in \Lambda(\ox,\ov)$. Pick $w\in K_f(\ox,\ov) $. Since $g$ is parabolically epi-differentiable at $F(\ox)$ for $\nabla F(\ox)w$ due to (H2),   Theorem~\ref{sochpar}(iii) implies that 
$f$ is parabolically epi-differentiable at $\ox$ for $w$, and so $\dom \d^2 f(\ox)(w\verl \cdot)\neq \emptyset$. This combined with  \eqref{pri3} and \eqref{sochpar1} results in \eqref{uess} and hence completes the proof.
\end{proof}

While  looking simple, the above result carries important information by which we can achieve a chain rule for the second subderivative.  
To do so, we should look for conditions under which the lower and upper estimates \eqref{less} and \eqref{uess}, respectively, 
coincide. This motivates us to consider the unconstrained optimization problem
\begin{equation}\label{prim}
\min_{z\in \X}\;\; -\la z, \ov\ra+ \d^2 g(F(\ox))( \nabla F(\ox) w \verlm \nabla F(\ox) z + \nabla^2 F(\ox)(w,w)). 
\end{equation}
When the basic assumptions (H1)-(H3) are  satisfied, \eqref{prim} is a convex optimization problem for any $w\in K_f(\ox,\ov)$. Using Proposition~\ref{dpd} 
allows us to obtain the dual problem of \eqref{prim} and then examine whether their optimal values coincide. 
We pursue this goal in the following result.

\begin{Theorem}[duality relationships]\label{duality}
Let $f:\X\to \oR$ have the composite representation \eqref{CF} at $\ox\in \dom f$,  $\ov\in \sub f(\ox)$, and 
let the basic assumptions {\rm(H1)-\rm(H3)} hold for $f$ at $(\ox,\ov)$.   Then for each $w \in K_f (\ox , \ov) $, the following assertions  are satisfied:  
\begin{itemize}[noitemsep,topsep=0pt]
\item [\bf{(i)}] the dual problem of   \eqref{prim} is given by 
\begin{equation}\label{dual}
\max_{y\in \Y}\;\;  \la y , \nabla^2 F(\ox) (w,w) \ra + \d^2 g(F(\ox) , y) ( \nabla F(\ox) w )\quad \mbox{subject to}\;\; y\in \Lambda(\ox,\ov);
\end{equation}
\item [\bf{(ii)}] the optimal values of the primal and dual problems \eqref{prim} and \eqref{dual}, respectively, 
 are finite and coincide; moreover, we have  $\Lambda (\ox , \ov , w) \cap (\tau  \B) \neq \emptyset$, where  $\Lambda (\ox , \ov , w)$ 
 stands for the set of optimal solutions to the dual problem \eqref{dual}
and where  
\begin{equation}\label{tau}
\tau := \kappa \ell  \| \nabla F(\ox)\| +  \kappa\| \ov \|+ \ell
\end{equation}
 with  $\ell$ and $\kappa$ taken from \eqref{CF} and \eqref{mscq1}, respectively.
\end{itemize}
\end{Theorem}
\begin{proof} Pick $w \in K_f (\ox , \ov) $  and observe from \eqref{cri2} that $\nabla F(\ox)w\in K_g(F(\ox),y)$ for all $y\in \Lambda(\ox,\ov)$. 
This together with Proposition~\ref{dpd} ensures that the parabolic subderivative of $g$ at $F(\ox)$ for $\nabla F(\ox)w$ is 
a proper, l.s.c., and convex function. 
Using this combined with  \cite[Example~11.41]{rw} and \eqref{conj} tells us that the dual problem of \eqref{prim} is 
$$ 
\max_{y\in \Y}\;\;  \la y , \nabla^2 F(\ox) (w,w) \ra + \d^2 g(F(\ox) , y) ( \nabla F(\ox) w )\quad \mbox{subject to}\;\; y\in \Lambda(\ox,\ov)\cap {\cal D},
$$
where ${\cal D}:=\{y\in \Y|\, \d g(F(\ox))(\nabla F(\ox)w)=\la y,\nabla F(\ox)w\ra\}$. Since $\nabla F(\ox)w\in K_g(F(\ox),y)$ for all $y\in \Lambda(\ox,\ov)$, we 
obtain 
$$
y\in \Lambda(\ox,\ov)\cap {\cal D} \iff y\in \Lambda(\ox,\ov).
$$
Combining these confirms that the dual problem of \eqref{prim} is equivalent to \eqref{dual} and thus finishes the proof of (i). 
To prove (ii), consider the optimal value function $\vartheta:\Y\to [-\infty,\infty]$, defined by 
\begin{equation}\label{perturb}
\vartheta(p)=\inf_{z \in \X} \big\{  -\la \ov , z \ra  + \d^2 g(F(\ox))( \nabla F(\ox) w \verll \nabla F(\ox) z + \nabla^2 F(\ox)(w,w) + p)  \big \},\quad p\in \Y.
\end{equation} 
We proceed with the following claim:

{\bf Claim.}{\em We have  $\sub \vartheta(0)\neq \emptyset$.}

To justify the claim, we first need to show $\vartheta(0)\in \R$. To do so, observe that $\ov\in \sub f(\ox)=\sub^p f(\ox)$ due to Proposition~\ref{fsch}(ii). 
Thus, it follows from Proposition~\ref{ssp}(iii) and \eqref{uess} that there is a constant $r\in \R_+$ such that for any $w\in K_f (\ox , \ov)$ we have 
$$
-r\|w\|^2 \le \d^2 f(\ox, \ov)(w)\le \vartheta(0)<\infty,
$$ 
which in turn implies that $\vartheta(0)\in \R$.  Next, we are going to show that 
\begin{equation}\label{nuf}
 \vartheta(p) \geq  \vartheta(0) - \tau \| p \|\quad \mbox{for all}\;\;p\in \X,
\end{equation}
where $\tau$ is taken from \eqref{tau}. 
To this end, take $(p,z) \in \Y\times \X$ such that  
\begin{equation*}\label{up}
 u_p : =\nabla F(\ox) z + \nabla^2 F(\ox)(w,w) + p \in  \dom \d^2 g (F(\ox)) ( \nabla F(\ox)w \verlm \cdot).  
\end{equation*}
 By \eqref{domgp}, we get $u_p\in T_{\ss \dom g}^2 (F(\ox) ,  \nabla F(\ox)w)$. Define now the set-valued mapping $S_w:\Y\tto\X$ by 
 \begin{equation*}
S_w (p) := \big\{ z \in \X |\;  \nabla F(\ox) z + \nabla^2 F(\ox)(w,w) + p \in T_{\ss\dom g}^2 (F(\ox) , \nabla F(\ox) w) \big \},\;\;p\in \Y.
\end{equation*} 
So, we get $z\in S_w(p)$. It was recently observed in \cite[Theorem~4.3]{mms2} that 
 the  mapping $S_w$ enjoys the uniform outer Lipschitzian property  at $0$ with constant $\kappa$ taken from \eqref{mscq1}, namely for every $p\in \Y$ we have 
\begin{equation*}\label{ulip}
S_w (p) \subset S_w (0) + \kappa \| p \| \B.       
\end{equation*}
This combined with  $z\in S_w(p)$ results in the existence of  $z_0 \in S_w(0)$ and $b \in \B$ such that  $z =  z_0 +  \kappa \|p\| b$. It follows from 
\eqref{domgp} and $z_0 \in S_w(0)$ that 
$$
\nabla F(\ox) z_0 + \nabla^2 F(\ox)(w,w)\in \dom \d^2 g (F(\ox)) ( \nabla F(\ox)w \verlm \cdot).  
$$ 
Since we have 
$$
u_p-\big(\nabla F(\ox) z_0 + \nabla^2 F(\ox)(w,w)\big)=p+\kappa \| p \| \nabla F(\ox) b,
$$
and since the parabolic subderivative $ \d^2 g (F(\ox)) ( \nabla F(\ox)w \verlm\cdot)$ is Lipschitz continuous relative to its domain due to Proposition~\ref{lcps}, 
we get the relationships 
\begin{eqnarray*}
-\la  \ov , z \ra + \d^2 g (F(\ox)) ( \nabla F(\ox)w \verlm u_p ) & \geq & -\la  \ov , z_0  \ra + \d^2 g (F(\ox))( \nabla F(\ox)w  \verlm \nabla F(\ox) z_0 + \nabla^2 F(\ox)(w,w)) \\  
&&- \ell  \,  \|  p+ \kappa \; \| p \| \nabla F(\ox) b    \; \| -  \kappa \|p\|\la \ov ,  b  \ra \\
& \geq &  \vartheta(0)  - \big( \ell \kappa  \| \nabla F(\ox)\| + \kappa  \| \ov \| +\ell\big) \| p \|,
\end{eqnarray*}
which together with \eqref{tau} justify \eqref{nuf}. Remember that the parabolic subderivative 
of $g$ at $F(\ox)$ for $\nabla F(\ox)w$ is a proper and convex function. This implies that the function 
$$
(z,p)\mapsto  -\la \ov , z \ra  + \d^2 g(F(\ox))( \nabla F(\ox) w \verll \nabla F(\ox) z + \nabla^2 F(\ox)(w,w) + p) 
$$
 is  convex on $\X\times \Y$. Using this together with \cite[Proposition~2.22]{rw} tells us  that $\vartheta $ is a convex function on $\Y$.
Thus, we conclude from  \eqref{nuf} and  \cite[Proposition~5.1]{mms2}  that there exists  a subgradient $\oy$ of $\vartheta$ at $0$ such that 
\begin{equation}\label{subg}
\oy\in \sub \vartheta(0)\cap ( \tau\B),
\end{equation} 
which completes the proof of the claim. 

Employing now \eqref{subg} and  \cite[Theorem~ 2.142]{bs} confirms that the optimal values of the primal and dual problems
\eqref{prim} and \eqref{dual}, respectively, coincide and that 
$$
\Lambda(\ox , \ov, w) = \sub \vartheta(0).
$$ 
This together with \eqref{subg} justifies (ii) and hence completes the proof.
\end{proof}

The above theorem extends the recent results obtained in \cite[Propositions~5.4 \& 5.5]{mms2} for constraint sets, namely when 
the function $g$ in \eqref{CF} is the indicator function of a closed convex set. We should add here that 
for constraint sets, the dual problem \eqref{dual} can be obtained via elementary arguments. However,
for the composite form \eqref{CF} a similar result requires using rather advanced theory of epi-convergence. 

\begin{Remark}[duality relationship under metric regularity]\label{remn}{\rm In the framework of Theorem \ref{duality}, we want to show that replacing 
assumption (H1) with the strictly stronger constraint qualification \eqref{gf07} allows us not only to drop the imposed Lipschitz continuity of $g$ from \eqref{CF}
but also to simplify the proof of Theorem \ref{duality}. To this end,   let $w \in K_f (\ox , \ov) $ and define the function 
$$
\psi(u):=  \d^2 g(F(\ox))( \nabla F(\ox) w \verlm  u),\;\; u\in \X.
$$
By Proposition~\ref{dpd}, $\psi$ is a proper, l.s.c., and convex function. Employing \cite[Proposition~13.12]{rw} tells us that 
\begin{equation}\label{dl8}
T_{\ss \epi g}(p)+T^2_{\ss \epi g}\big(p,q)\subset T^2_{\ss \epi g}\big(p,q)=\epi \psi ,
\end{equation}
where $p:=\big(F(\ox), g(F(\ox))\big)$ and $q:=\big(\nabla F(\ox)w, \d g(F(\ox))(\nabla F(\ox) w )\big)$ and where the equality in the right side comes from \cite[Example~13.62(b)]{rw}.
We are going to show that the validity of \eqref{gf07} yields 
\begin{equation}\label{dl9}
N_{\ss\dom \psi}(u)\cap \ker \nabla F(\ox)^*=\{0\}
\end{equation}
for any $u\in \dom \psi$. To this end, pick $u\in \dom \psi$ and conclude from \eqref{dl8} and \cite[Exercise~6.44]{rw} that 
$$
N_{\ss \epi \psi}\big(u,\psi (u)\big)\subset N_{\ss T_{\ss {\epi g}}(p)}(0)\cap N_{\ss \epi \psi}\big(u,\psi (u)\big)= N_{\ss \epi g}(p)\cap N_{\ss \epi \psi}\big(u,\psi (u)\big).
$$
This together with \eqref{gf07} and the relationship $N_{\ss\dom \psi}(u)=\sub^{\infty}\psi (u)$ stemming from the convexity of $\psi$ confirms the validity of \eqref{dl9}.  
Appealing now to \cite[Theorem~2.165]{bs} gives another proof of Theorem \ref{duality} when assumption (H1) therein is replaced with the strictly stronger constraint qualification \eqref{gf07}.
}
\end{Remark}

 The established duality relationships in  Theorem \ref{duality} open the door to derive a chain rule for parabolically regular functions and to 
  find an exact chain rule for the second subderivative of the composite function \eqref{CF} under our basic assumptions. 
  \begin{Theorem}[chain rule for parabolic regularity]\label{chpr}
Let $f:\X\to \oR$ have the composite representation \eqref{CF} at $\ox\in \dom f$,  $\ov\in \sub f(\ox)$, and 
let the basic assumptions {\rm(H1)-\rm(H3)} hold for $f$ at $(\ox,\ov)$. 
Then $f$ is  parabolically  regular  at $\ox$ for $\ov$. Furthermore, for every  $w \in\X$, the second subderivative of $f$ at $\ox$ for $\ov$
is calculated by 
\begin{eqnarray}
\d^2 f(\ox , \ov)(w) &=& \max_{y \in \Lambda(\ox , \ov)} \big\{  \la y , \nabla^2 F(\ox) (w,w) \ra + \d^2 g(F(\ox) , y) ( \nabla F(\ox) w ) \big \}\label{chss}\\
&=& \max_{y \in \Lambda(\ox , \ov)\,\cap\, (\tau\B)} \big\{  \la y , \nabla^2 F(\ox) (w,w) \ra + \d^2 g(F(\ox) , y) ( \nabla F(\ox) w ) \big \},\nonumber
\end{eqnarray}
where $\tau$ is taken from \eqref{tau}.
\end{Theorem}
\begin{proof} It was recently observed in \cite[Corollary~3.7]{mms1} that the Lagrange multiplier set $\Lambda(\ox,\ov)$ enjoys the following property:
\begin{equation}\label{lagf}
\Lambda(\ox , \ov)\,\cap\, (\tau\B)\neq \emptyset.
\end{equation}
Take $w\in K_f(\ox,\ov)$. By \eqref{less} and Theorem~\ref{duality}(ii), we obtain 
\begin{equation}\label{gc01}
 \max_{y \in \Lambda(\ox , \ov)} \big\{  \la y , \nabla^2 F(\ox) (w,w) \ra + \d^2 g(F(\ox) , y) ( \nabla F(\ox) w ) \big \}\le \d^2 f (\bar x , \ov ) (w).
\end{equation}
On the other hand, using \eqref{pri3}, \eqref{sochpar1}, and Theorem~\ref{duality}(ii), respectively,  gives us the inequalities  
\begin{eqnarray*}
\d^2 f (\bar x , \ov ) (w)  & \leq &  \inf_{z\in \X}\big\{\d^2 f(\ox)(w\verl z) -\la z, \ov\ra\big \}\\
 &=&\inf_{z \in \X} \big\{  -\la z, \ov \ra  + \d^2 g(F(\ox))( \nabla F(\ox) w \verlm \nabla F(\ox) z + \nabla^2 F(\ox)(w,w)) \big \} \\
&=&  \max_{y \in \Lambda(\ox , \ov)\,\cap\, (\tau\B)} \big\{  \la y , \nabla^2 F(\ox) (w,w) \ra + \d^2 g(F(\ox) , y) ( \nabla F(\ox) w ) \big \}.
\end{eqnarray*} 
These combined with  \eqref{gc01} ensure that the  claimed second subderivative formulas  for $f$ at $\ox$ for $\ov$ hold for any $w\in K_f(\ox,\ov)$ and 
that 
$$
\d^2 f (\bar x , \ov ) (w)  =  \inf_{z\in \X}\big\{\d^2 f(\ox)(w\verl z) -\la z, \ov\ra\big \}\quad \mbox{for all}\;\; w\in K_f(\ox,\ov).
$$
Appealing  now to Proposition~\ref{parregularity}, we conclude that  $f$ is parabolically regular at $\ox$ for $\ov$. 

What remains is to validate  the second subderivative formulas for $w\notin K_f(\ox,\ov)$. It follows from Theorem~\ref{sochpar}(iii)
that $f$ is parabolically epi-differentiable at $\ox$ for every $w\in K_f(\ox,\ov)$ and thus $\dom \d^2 f(\ox)(w\verl \cdot)\neq \emptyset$
for every $w\in K_f(\ox,\ov)$. So, by Proposition~\ref{dss} we have $\dom \d^2 f(\ox, \ov)=K_f(\ox,\ov)$. Since the second subderivative 
$\d^2 f(\ox, \ov)$ is a proper function, we obtain $\d^2 f(\ox, \ov)(w)=\infty$ for all $w\notin K_f(\ox,\ov)$. On the other hand, 
we understand from \eqref{cri2} that $w\notin K_f(\ox,\ov)$ amounts to $\nabla F(\ox)w\notin K_g(F(\ox),y)$ for every $y\in \Lambda(\ox,\ov)$.
Combining the basic assumption (H2) and Proposition~\ref{dss} tells us that for every $y\in \Lambda(\ox,\ov)$ we have $ \d^2 g(F(\ox) , y) ( \nabla F(\ox) w ) =\infty$
whenever $w\notin K_f(\ox,\ov)$. This together with \eqref{lagf} confirms that both sides in \eqref{chss} are $\infty$ for every $w\notin K_f(\ox,\ov)$ and thus 
the claimed formulas for the second subderivative of $f$ hold for this case. This completes the proof.
\end{proof}

A chain rule for parabolic regularity of the composite function \eqref{CF}, where $g$ is not necessarily locally Lipschitz continuous relative to its domain, 
was established in \cite[Proposition~3.104]{bs}. The assumptions utilized in the latter result were stronger than those used in Theorem~\ref{chpr}.
Indeed,  \cite[Proposition~3.104]{bs} assumes that $g$ is second-order regular in the sense of \cite[Definition~3.93]{bs} and the basic constraint qualification  \eqref{gf07}
is satisfied  and uses a different approach  to derive this result.  When $g$ is a convex piecewise linear-quadratic, parabolic regularity of the composite function  \eqref{CF}
was established in \cite[Theorem~13.67]{rw} under the stronger condition \eqref{gf07}. Theorem~\ref{chpr} covers the aforementioned results and shows that 
we can achieve a similar conclusion under the significantly weaker condition \eqref{mscq1}.

As an immediate consequence of the above theorem, we can easily guarantee the twice epi-differentiability of the composite form \eqref{CF}
under our basic assumptions.
\begin{Corollary}[chain rule for twice epi-differentiability]\label{tedc}  Let the function $f$ from \eqref{CF} satisfy all the assumptions of Theorem~{\rm\ref{chpr}}.
Then $f$ is twice epi-differentiable at $\ox$ for $\ov$.
\end{Corollary}
\begin{proof} By Theorem~\ref{sochpar}(iii), $f$ is parabolically epi-differentiable at $\ox$ for every $w\in K_f(\ox,\ov)$. 
Employing now Theorems~\ref{chpr} and \ref{tedp} implies that $f$ is twice epi-differentiable at $\ox$ for $\ov$.
\end{proof}

\begin{Remark}[discussion on twice epi-differentiability]\label{rem01}{\rm 
 Corollary~\ref{tedc} provides a far-going extension of the available results for the twice epi-differentiability 
of extended-real-valued functions. To elaborate more, suppose that $f:\X\to \oR$ has a composite form \eqref{CF} at $\ox\in \dom f$.
Then the following observations hold:
\begin{itemize}[noitemsep,topsep=1pt]
\item [\bf{(a)}] If $\X=\R^n$,  $\Y=\R^m$,   and $g$ in \eqref{CF} is convex piecewise linear-quadratic, then Rockafellar proved in \cite{r88} that 
under the fulfillment of the basic constraint qualification  \eqref{gf07}, $f$ is twice epi-differentiable. This result was improved recently in \cite[Theorem~5.2]{mms1},
where it was shown that using the strictly weaker condition \eqref{mscq1} in the Rockafellar's framework \cite{r88} suffices to ensure 
the twice epi-differentiability of $f$. Taking  into account Example~\ref{ex04}(a)  tells us both these results can be derived from Corollary~\ref{tedc}.

\item [\bf{(b)}] If $\X=\R^n$, $\Y=\S^m$,  and  $g$  is either the maximum eigenvalue function $\lm_{\max}$ from \eqref{maxeig} or the function $\sigma_i$ from \eqref{fsig},  then we fall 
into the framework considered by Turki in  \cite[Theorems~2.3 \& 2.5]{t2} in which  he justified the twice epi-differentiability of  $f$. Since in this framework we have $\dom g=\S^m$, both conditions \eqref{gf07} and \eqref{mscq1}
are automatically satisfied. By Example~\ref{ex04}(b), the twice epi-differentiability of $f$ can be deduced from Corollary~\ref{tedc}.
 
\item [\bf{(c)}] If $\X=\R^n$,  $\Y=\R^m$, and $g=\dd_C$ with the closed convex set $C$ taken from Example~\ref{ex04}(c), 
we fall into the framework considered in \cite{mms2}. In this case, Corollary~\ref{tedc} can cover the twice epi-differentiability of $f$ obtained in \cite[Corollary~5.11]{mms2}.

\item [\bf{(d)}] If $\X=\R^n$,  $\Y=\R^m$, and $g$ is a proper, convex, l.s.c., and positively homogeneous, then 
we fall into the framework, considered by Shapiro in \cite{sh03}. In this case, the composite form \eqref{CF}
is called {\em decomposable}; see \cite{sh03,mi} for more detail about this class of extended-real-valued functions. 
It was proven in \cite[Lemma~5.3.27]{mi} that for this case of $g$, the composite form \eqref{CF} is twice epi-differentiable 
if it is convex and if the nondegeneracy condition for this setting holds; see \cite[Definition~5.3.1]{mi} for the definition of this condition. 
In this framework, by the positive homogeneity of $g$ and $F(\ox)=0$, coming from \cite[Definition~5.3.1]{mi},   we can easily  show that $g$ is parabolically regular. 
Moreover the  assumed nondegeneracy condition in \cite[Lemma~5.3.27]{mi} yields the validity of condition \eqref{gf07}. 
As pointed out in Remark~\ref{remn}, the Lipschitz continuity of $g$ in the composite form \eqref{CF} can be relaxed when condition \eqref{gf07} is satisfied. 
 Since the nondegeneracy condition implies that the set of Lagrange multipliers $\Lm(\ox,\ov)$ is a singleton, 
 we can use estimates \eqref{less} and \eqref{uess} to justify 
 parabolic regularity of the composite form \eqref{CF} in the framework of \cite{mi}. This together with Corollary~\ref{tedc}
 allows to recover \cite[Lemma~5.3.27]{mi}. Furthermore, we can drop the convexity of the composite form \eqref{CF},
 assumed in \cite{mi}.
\end{itemize}
 }
\end{Remark}

\vspace{-0.2 in}
\section{Second-Order Optimality Conditions for Composite Problems}\sce \label{sect06}

In this section, we focus mainly on obtaining second-order optimality conditions for the composite problem 
\eqref{comp}, where $\ph:\X\to \R$ and $F:\X\to \Y$ are   twice differentiable    and the function $g:\Y\to \oR$ is 
an  l.s.c.\ convex function that is locally Lipschitz continuous relative to  its domain.  The latter means that for any $y\in \dom g$,
the function $g$ is Lipschitz continuous around $y$ relative to  its domain. 
Important examples of constrained and composite optimization problems can be achieved when $g$ is one of the functions considered in Example~\ref{ex04}. 
For any pair $(x,y)\in \X\times \Y$, the Lagrangian associated with the composite 
problem \eqref{comp} is defined by 
$$
L(x,y)=\ph(x)+\la F(x), y\ra-g^*(y),
$$
where $g^*$ is the Fenchel conjugate of the convex function $g$. We begin with the following result in which 
we collect second-order optimality conditions for \eqref{comp} when our basic assumptions are satisfied. Recall 
that a point $\ox\in \X$ is called a feasible solution to the composite problem \eqref{comp} if we have $F(\ox)\in \dom g$. 
\begin{Theorem}[second-order optimality conditions]\label{sooc}  
Let $\ox$ be a feasible solution to problem \eqref{comp} and let $f:=g\circ F$ 
and $\ov:=-\nabla \ph(\ox)\in \sub f(\ox)$ with $\ph$, $g$, and $F$ taken from \eqref{comp}.  Assume that the basic assumptions {\rm (H1)-(H3)} hold for $f$ at $(\ox,\ov)$.
Then  the following second-order optimality conditions for the composite problem \eqref{comp} are satisfied: 
\begin{itemize}[noitemsep, topsep=1pt]
\item [\bf{(i)}]  if $\ox$ is  a local minimum of \eqref{comp}, then the  second-order necessary condition 
\begin{equation*}\label{nopc1}
 \max_{y \in  \Lambda(\ox,\ov)} \big\{\langle \nabla_{xx}^2L(\bar x,y)w,w\rangle+\d^2 g(F(\bar x), y)(\nabla F(\ox)w)\big\}  \geq 0 
\end{equation*}
holds for all $w\in K_f(\ox,\ov)$;
\item [\bf{(ii)}] the validity of   the second-order condition
\begin{equation}\label{sscc}
 \max_{y \in  \Lambda(\ox,\ov)} \big\{\langle \nabla_{xx}^2L(\bar x,y)w,w\rangle+\d^2 g(F(\bar x), y)(\nabla F(\ox)w)\big\} > 0 \quad \mbox{for all}\quad  w \in K_f(\ox,\ov)\setminus \{0\}
\end{equation}
amounts to the existence of  constants $\ell>0$ and  $\ve > 0$ such that the second-order growth condition
\begin{equation}\label{grow}
\psi(x) \ge \psi(\ox) + \lsm \|x - \ox \|^2 \quad \mbox{for all}\;\;x \in  \B_{\ve} (\ox)
\end{equation}
holds, where $\psi:=\ph+g\circ F$.
\end{itemize}

\end{Theorem}
\begin{proof}  To justify (i), note that since $\ox$ is  a local minimum of \eqref{comp}, it is a local minimum of $\psi=\ph+f$. 
Moreover, $-\nabla \ph(\ox)\in \sub f(\ox)$ amounts to $0\in \sub \psi(\ox)$. Thus, by definition, we arrive at $\d^2 \psi (\ox,0)(w)\ge 0$ for all $w\in \X$. Since $\ph$ is twice differentiable at $\ox$, we
obtain the following sum rule for the second subderivatives: 
\begin{equation}\label{sumr}
\d^2 \psi (\ox,0)(w)=\la \nabla^2 \ph(\ox)w,w\ra +\d^2 f(\ox,\ov)(w)\quad \mbox{for all}\;\; w\in \X.
\end{equation}
Combing these with the chain rule \eqref{chss} proves (i).

Turing now to (ii), we infer from \cite[Theorem~13.24(c)]{rw} that $\d^2 \psi (\ox,0)(w)> 0$ for all $w\in \X\setminus\{0\}$ amounts to the existence 
of some constants  $\ell>0$ and  $\ve > 0$ for which the second-order growth condition \eqref{grow} holds. Remember from \eqref{domss} and \eqref{sumr} that 
\begin{equation}\label{domps4}
\dom \d^2 \psi (\ox,0)=\dom \d^2 f (\ox,\ov)=K_f(\ox,\ov).
\end{equation}
Using these,  the chain rule \eqref{chss},  and  the sum rule \eqref{sumr} proves the claimed equivalence in (ii) and thus finishes the proof.
\end{proof}

\begin{Remark}[discussion on second-order optimality conditions] {\rm The second-order optimality conditions 
for composite problems were established in \cite[Theorems~3.108 \& 3.109]{bs} for \eqref{comp}
by expressing \eqref{comp} equivalently  as a constrained problem and then appealing to the theory of  second-order 
optimality conditions for  the latter class of problems. While not assuming that $g$ is locally Lipschitz continuous relative to its 
domain, theses results were established under condition \eqref{gf07} and the second-order regularity in the sense of \cite[Definition~3.93]{bs}
that are strictly stronger than condition \eqref{mscq1} and the parabolic regularity, respectively, we imposed in Theorem~\ref{sooc}.
Another major difference is that we require that  $g$ be parabolically epi-differentiable (assumption (H2)), which was 
not assumed in \cite{bs}. This  assumption plays an  important role in our developments and has two important consequences:
 1) it makes the parabolic subderivative be a {\em convex} function and help us obtain a precise formula for  the Fenchel conjugate  
 of the parabolic subderivative  in our framework; 2) it  allows to establish the equivalence between  \eqref{sscc} and the growth condition in Theorem~\ref{sooc}.
These facts were not achieved in \cite{bs}; indeed,   \cite[Theorem~ 3.109]{bs} was written in terms of the conjugate of  the parabolic subderivative
 and only states that  condition  \eqref{sscc} implies the growth condition therein.As discussed in Remark~\ref{remn}, if we replace condition \eqref{mscq1}  with the stronger condition \eqref{gf07}, the imposed Lipschitz 
continuity of $g$ can be relaxed in our developments. 
 It is worth mentioning that the imposed Lipschitz 
 continuity of $g$ relative to its domain, utilized in this paper, does not seem to be restrictive and allows us to provide an umbrella 
 under which second-order variational analysis for  composite problems can be carried out under condition \eqref{gf07} in the same level of perfection as those 
 for constrained problems. We believe that if we strengthen condition \eqref{gf07} to the metric subregularity of the epigraphical  mapping  $(x,\al)\mapsto (F(x),\al)-\epi g$, 
 the imposed  Lipschitz continuity of $g$ can be relaxed in our developments.

 Cominetti  \cite[Theorem~5.1]{c91} established second-order optimality conditions for the composite problem \eqref{comp} 
 similar to Theorem~\ref{sooc} without making a connection between \eqref{sscc} and the growth condition \eqref{grow}. As mentioned in our 
 discussion after Example~\ref{ex04}, the results in \cite{c91} were established  under condition \eqref{gf07} and a restrictive assumption on the second
 subderivative, which does not hold for important classes of composite problems. When we are in the framework of Remark~\ref{rem01}(a),
 Theorem~\ref{sooc} was first achieved by Rockafellar in \cite[Theorem~4.2]{r89} under condition \eqref{gf07} and was improved recently 
 in \cite[Theorem~6.2]{mms1} by replacing the latter condition with \eqref{mscq1}.  For the framework of Remark~\ref{rem01}(b), 
 the second-order optimality conditions from Theorem~\ref{sooc} were obtained in \cite[Theorem~4.2]{t2}. Finally, if we are in the framework of 
 Remark~\ref{rem01}(c), Theorem~\ref{sooc} covers our recent developments in \cite{mms2}. 
 }
\end{Remark}

We end this section by  obtaining a characterization of strong metric subregularity of the subgradient mapping 
of the objective function of the composite problem  \eqref{comp}. 
\begin{Theorem}[strong metric subregularity of the subgradient mappings in composite problems]\label{smse}
Let $\ox$ be a feasible solution to problem \eqref{comp} and let $f:=g\circ F$
and $\ov:=-\nabla \ph(\ox)\in \sub f(\ox)$  with $\ph$, $g$, and $F$ taken from \eqref{comp}.  Assume that the basic assumptions {\rm (H1)-(H3)} hold for $f$ at $(\ox,\ov)$
and that both $\ph$ and $F$ are ${\cal C}^2$-smooth around $\ox$. 
Then the following conditions are equivalent:
\begin{itemize}[noitemsep,topsep=1pt]
\item [\bf{(i)}] the point $\ox$ is a local minimizer for $\psi=\ph+f$  and the subgradient mapping $\sub \psi$ is strongly metrically subregular at $(\ox,0)$;
\item [\bf{(ii)}] the second-order sufficient condition \eqref{sscc} holds. 
\end{itemize}
\end{Theorem}
\begin{proof} We conclude from \eqref{sumr} and \eqref{domps4} that \eqref{sscc} amounts to the fulfillment of the condition 
\begin{equation}\label{posi}
\d^2 \psi(\ox, 0)(w)>0\quad \mbox{for all}\;\; w\in \X\setminus \{0\}.
\end{equation}
If (i) holds, we conclude from the local optimality of $\ox$ that $\d^2 \psi(\ox, 0)(w)\ge 0$ for all $w\in \X$. 
Since  (ii) is equivalent to \eqref{posi}, it suffices to show that there is no $w\in \X\setminus \{0\}$ such that $\d^2 \psi(\ox, 0)(w)= 0$.
Suppose on the contrary that there exists  $\ow\in \X\setminus\{0\}$ satisfying the latter condition. This means that $\ow$ is a minimizer for the problem 
\begin{equation*}
\mbox{minimize}\;\sm\d^2 \psi(\ox, 0)(w) \quad\mbox{subject to  }\;\;w\in \X.
\end{equation*}
Since    both $\ph$ and $F$ are ${\cal C}^2$-smooth around $\ox$, we 
can show using similar arguments as \cite[Proposition~7.1]{mms1} that $\psi$ is prox-regular and subdifferentially continuousat $\ox$ for $0$. 
This  together with the Fermat stationary principle and \eqref{gdpd} results in 
\begin{eqnarray}\label{ferm}
0\in \sub\big(\sm\d^2 \psi(\ox, 0)\big)(\ow)=D(\sub \psi)(\ox, 0)(\ow).
\end{eqnarray}
Since $\sub \psi$ is strongly metrically subregular at $(\ox,0)$, we deduce from \eqref{sms8} that $\ow=0$, a contradiction. This proves (ii).

To justify the opposite implication, assume that (ii) holds. According to Theorem~\ref{sooc}(ii), $\ox$ is a local minimizer for $\psi$. 
Pick now $w\in \X$ such that $0\in D(\sub \psi)(\ox, 0)(w) $. To obtain (i), we require by \eqref{sms8} to show that $w=0$.
Employing now \eqref{ferm} yields $0\in \sub\big(\sm\d^2 \psi(\ox, 0)\big)(w)$. This combined with  \cite[Lemma~3.7]{chnt} confirms that  $\d^2 \psi(\ox, 0)(w)=\la 0,w\ra=0$. 
Remember that  (ii) is equivalent to \eqref{posi}. Combining these  results in  $w=0$ and thus proves (i).
\end{proof}

The above result was first observed in \cite[Theorem~4G.1]{dr} for a subclass of nonlinear programming problems and was 
extended in \cite[Theorem~4.6]{chnt} for ${\cal C}^2$-cone reducible constrained optimization problems and  in \cite[Theorem~9.2]{mms2} 
for parabolically regular constrained optimization problems.  The theory of the twice epi-differentiability, obtained in this paper, 
provides an easy path to achieve a similar result for the composite problem \eqref{comp}. 

It is worth mentioning that similar characterizations as \cite[Theorem~4.2]{mms2}
can be achieved for the KKT system of \eqref{comp}. Furthermore,
Corollary~\ref{pdpr} provides a systematic method to calculate proto-derivatives of subgradient mappings of functions enjoying the composite form \eqref{CF},
a path we will pursue in our future research.

{\bf Acknowledgements.} The second author would like to thank  Asen Dontchev for 
bringing \cite[Theorem~4G.1]{dr} to his attention, which inspired the equivalence  obtained in  Theorem~\ref{smse}. Special thanks go to  Boris Mordukhovich for  numerous discussions on the subject of this paper and second-order variational analysis
 and for the insightful advice and comments. 
We also thank two anonymous reviewers for their  useful comments that allowed us to improve the original presentation. Reference \cite{mi} was brought to our attention by one of the referees that is highly appreciated.
\vspace*{-0.1in}



\end{document}